\documentclass[11pt]{amsart}
\usepackage{graphicx}
\usepackage{amssymb}
\usepackage{epstopdf}
\usepackage{amsfonts}
\usepackage{amsmath}
\usepackage{esint,color}
\usepackage{amsthm}
\usepackage{enumerate}
\usepackage{mathtools}

\newcommand{\R}{{\mathbb R}}

\newcommand{\N}{{\mathbb N}}

\newcommand{\cS}{{\mathcal S}}

\newcommand{\cB}{{\mathcal B}}

\newcommand{\cN}{{\mathcal N}}

\newcommand{\e}{\varepsilon}
\newcommand{\vp}{\varphi}
\newcommand{\osc}{\operatornamewithlimits{osc}}

\newcommand{\ddiv}{\operatorname{div}}
\newcommand{\diam}{\operatorname{diam}}
\newcommand{\dist}{\operatorname{dist}}

\newcommand{\tr}{\operatorname{tr}}
\newcommand{\ra}{\rightarrow}

\newcommand\norm[1]{\left\Arrowvert {#1} \right\Arrowvert}

\theoremstyle{plain}
\newtheorem{theorem}{Theorem}[section]

\newtheorem{corollary}[theorem]{Corollary}

\newtheorem{lemma}[theorem]{Lemma}
\newtheorem{proposition}[theorem]{Proposition}

\theoremstyle{definition}

\theoremstyle{remark}

\newtheorem{remark}[theorem]{Remark}

\numberwithin{equation}{section}

\title[Nodal Sets for Broken Quasilinear Equations]{Nodal Sets for Broken Quasilinear Partial Differential Equations with Dini Coefficients}
\author{Sunghan Kim}
\address{Department of Mathematics, KTH Royal Institute of Technology, 100 44 Stockholm, Sweden}
\email{sunghan@kth.se}
\thanks {This project was supported by Knut and Alice Wallenberg Foundation.}

\begin{document}
\maketitle


\begin{abstract}
This paper is concerned with the nodal set of weak solutions to a broken quasilinear partial differential equation, 
\begin{equation*}
\ddiv (a_+ \nabla u^+ - a_- \nabla u^-) = \ddiv f,
\end{equation*}
where $a_+$ and $a_-$ are uniformly elliptic, Dini continuous coefficient matrices, subject to a strong correlation that $a_+$ and $a_-$ are a multiple of some scalar function to each other. Under such a structural condition, we develop an iteration argument to achieve higher-order approximation of solutions at a singular point, which is also new for standard elliptic PDEs below H\"older regime, and as a result, we establish a structure theorem for singular sets. We also estimate the Hausdorff measure of nodal sets, provided that the vanishing order of given solution is bounded throughout its nodal set, via an approach that extends the classical argument to certain solutions with discontinuous gradient. Besides, we also prove Lipschitz regularity of solutions and continuous differentiability of their nodal set around regular points.
\end{abstract}

\tableofcontents


\section{Introduction}\label{section:intro}

This article is concerned with nodal set of weak solutions to a broken\footnote{The term ``broken'' was first introduced in an earlier collaboration \cite{KLS2} by Lee, Shahgholian and the author, to highlight the fact that $u$ can be viewed as a regular solution with a break across its nodal set $\{u=0\}$; for instance, in the simplest setting where $a_+ = I$ and $a_- = \kappa_0 I$ for some real constant $\kappa_0$, $u^+ - \kappa_0 u^-$ becomes a harmonic function.
} quasilinear partial differential equation (PDE in the sequel), 
\begin{equation}\label{eq:main-gen}
\ddiv(a_+\nabla u^+ - a_-\nabla u^-) = \ddiv f,
\end{equation} 
where $a_+$ and $a_-$ are uniformly elliptic matrices that are scalar multiple to each other, i.e., $a_- = \kappa a_+$ for some real-valued function $\kappa$; here $u^+ = \max\{u,0\}$ and $u^- = -\min\{u,0\}$. By nodal set, we shall indicate the zero-level set, $\{u=0\}$, of given solution. 

The strong correlation between $a_+$ and $a_-$ is assumed to obtain Lipschitz regularity of weak solutions to \eqref{eq:main-gen}. In fact, the solutions are not Lipschitz regular in general, without the correlation between the coefficients, even when both $a_+$ and $a_-$ are constant matrices. Here we assume this relation in order to study finer properties of the nodal set of the solutions. 

Nevertheless, it is challenging to analyse the structure of singular set\footnote{The set where solutions vanish faster than any linear function.} or estimate the (lower dimensional) measure of the nodal set, when the solutions are only Lipschitz continuous. Many problems ranging from general elliptic/parabolic PDEs to certain free boundary problems are studied with $C^{1,\alpha}$-solutions, accompanied with some strong tools, such as Almgren's or Weiss' type monotonicity formulae. Unfortunately, neither of them are available for weak solutions to \eqref{eq:main-gen}, even under the assumption that $a_+$ and $a_-$ are scalar multiple to each other. 

The purpose of this paper is to extend the classical geometric measure estimate of the nodal set, derived by Hardt and Simon \cite{HS} (which was later refined by Han and Lin \cite{HL94}), and the classical structure theorem for its singular part, as in \cite{Han}, when given solutions do not have continuous gradient across the nodal set. Our approach will be based on an observation of a particular structure in the break of the gradient, induced by the strong correlation between coefficients $a_+$ and $a_-$. 

Concerning the geometric measure of the nodal set of weak solutions to \eqref{eq:main-gen}, a key observation we make here is that although $\nabla u$ is only bounded and not continuous across the nodal set, mapping $\nabla u^+ - \kappa \nabla u^-$ (with $\kappa$ being the scalar function for which $a_- = \kappa a_+$) admits a uniformly continuous version, $Vu$ (see Theorem \ref{theorem:int-Lip}), such that if $u^+ - \kappa u^-$ approximates a harmonic function, then $Vu$ approximates the gradient of the harmonic function (Lemma \ref{lemma:compact}) in the sup-norm, from which a nodal set comparison lemma (Lemma \ref{lemma:compare}) for solutions to broken quasilinear PDEs follows. For this reason, we are able to follow the main idea of \cite{HS} and \cite{HL94}, and establish the desired measure estimate of the nodal set. Let us remark that our result (Theorem \ref{theorem:meas}) is also new in the context of standard elliptic PDEs without jump discontinuity, as we are concerned with divergence-type PDEs with Dini coefficients. Nevertheless, it does not require serious modification to extend the argument in \cite{HS} and \cite{HL94} to Dini regime, if there is no break in the coefficient. Instead, we claim the novelty of our measure estimate in the aspect that it considers discontinuous coefficients, and overcomes the difficulties arising from the break in the gradient. 

Our analysis on the singular set resembles an earlier collaboration \cite{KLS2} between Lee, Shahgholian and the author, in the sense that here we also establish uniform point-wise approximation (Lemma \ref{lemma:int-C1} and Lemma \ref{lemma:int-Cd}) for the corrected version, $u^+ - \kappa(z) u^-$, for each point $z$. However, we develop a new iteration argument (Proposition \ref{proposition:v-C1} and Proposition \ref{proposition:v-Cd}) to perform the uniform approximation, since our coefficients $a_+$ and $a_-$ are assumed to be Dini continuous, while the argument involved in \cite{KLS2} is only available for H\"older coefficients. 

Let us stress that our iteration scheme for higher-order approximation (Proposition \ref{proposition:v-Cd}) is also new in the context of standard (i.e., $a_+ = a_-$) elliptic PDEs, since the existing results, e.g., \cite{Bers}, \cite{Han}, \cite{Chen} and many others, are concerned with H\"older coefficients. As a result, we establish a structure theorem (Theorem \ref{theorem:sing}) for the singular set, and to the best of the author's knowledge, it is the first result available below the H\"older regime, even when there is no jump discontinuity of the coefficients. 

Another interesting observation regarding our structure theorem is that it also improves the classical result \cite{Han} by Han, when we recover it to H\"older coefficients. This is because Theorem \ref{theorem:sing} asserts that not only the singular set of dimension $n-2$ (with $n$ the space dimension) but also those with lower dimensions are equally contained in a countable union of $C^{1,\alpha}$-manifolds. Such an improvement has interests on its own as it implies that what is known to be the ``bad'' part of the singular set, i.e., the part with dimension strictly less than $n-2$, actually turns out to be a ``good'' part as well. 

For the rest of this section, let us illustrate the background of broken quasilinear PDEs \eqref{eq:main-gen}, and present some recent development around this problem.  

One can understand \eqref{eq:main-gen} as an elliptic free boundary problem arising from jump of conductivity, which models a composite material that undergoes discontinuous conductivity change subject to transition of one phase to another; for more concrete examples as well as the existence of weak solutions, we refer to \cite{KLS1}. Here the free boundary is the nodal set $\{u = 0\}$. As in \cite{AM13}, one can also derive \eqref{eq:main-gen} from a solution to a simple Bellman equation subject to two strategies, i.e., $\min\{\tr(a_+D^2 w),\tr(a_- D^2 w)\} = 0$. In this case, $u = \tr (a_+ D^2 w)$ becomes a weak solution to \eqref{eq:main-gen} (obviously with $\ddiv f = 0$), and the free boundary, $\{u=0\}$, becomes the threshold of switching strategies governed by $a_+$ and respectively $a_-$. Another variant closely related to \eqref{eq:main-gen} is the free transmission problem, considered by Amaral and Teixeira \cite{AT15}, in which case $u$ is given, instead, as a minimiser of functional $\int ( a_+|\nabla u^+|^2 + a_- |\nabla u^-| ^2 )\,dx$; heuristically, a first variation leads us to \eqref{eq:main-gen} with $\ddiv f$ replaced by a measure supported on the free boundary $\{u=0\}$.  

Let us overview some recent progress of \eqref{eq:main-gen} from the standpoint of free boundary problems. Without the key assumption that $a_+$ and $a_-$ are scalar multiple to each other, there has been only a few literature concerning the free boundary stemming from \eqref{eq:main-gen}, even in the simplest scenario that $a_+$, $a_-$ and $f$ are all constant (so that $\ddiv f = 0$). As mentioned briefly in the beginning, one of the main reasons is that the solutions are not Lipschitz regular in general, since there are homogeneous solutions that are only H\"older continuous, for every dimension larger than or equal to $3$. Only recently, Caffarelli, De Silva and Savin \cite{CDS18} proved that the Lipschitz regularity is true in space dimension $2$, without such an additional condition. Another reason is that the monotonicity formulae, such as the Alt-Caffarelli-Friedman functional, or balanced-energy functionals of Weiss' type, tend to fail as well, making it difficult to analyse the local properties of the free boundary via global solutions. As a matter of fact, global solutions to \eqref{eq:main-gen} are also interesting objects that deserve attention, although not much has been developed yet. 

It is worthwhile to address that Andersson and Mikayelyan \cite{AM13} established some meaningful results with $a_+$, $a_-$ and $f$ constant, but without the assumption that $a_+$ and $a_-$ are correlated by scalar multiplication. Owing to the lack of Lipschitz regularity of solutions, a novel approach was taken to ensure the flatness of the free boundary, $\{u=0\}$, that involved the measure-theoretic property of $\ddiv (a_+ \nabla u^+)$. With the flatness at hand, the $C^{1,\alpha}$-regularity of the free boundary around a regular point was achieved with a slight modification of the classical bootstrap argument \cite{CS05} for Bernoulli-type free boundary problems.

Later in \cite{KLS1}, Lee, Shahgholian and the author considered \eqref{eq:main-gen} with variable coefficients $a_+$ and $a_-$, and constant $f$ (so that $\ddiv f = 0$) and proved Lipschitz regularity of weak solutions to \eqref{eq:main-gen} under the aforementioned structural condition on $a_+$ and $a_-$, as well as Dini continuity of the coefficients. The key ingredient there was the monotonicity of the Alt-Caffarelli-Friedman functional, which was available particularly due to the additional correlation between $a_+$ and $a_-$. Then the authors proved $C^{1,\alpha}$-regularity of the free boundary around regularity points, partly following the argument of Andersson and Mikayelyan \cite{AM13}, provided that the coefficients are H\"older continuous. Although the right hand side was considered to be identically zero in \cite{KLS1}, the method can be easily extended to the case when $\ddiv f \in L^\infty$, i.e., $f\in C^{0,1}$. 

In a more recent collaboration \cite{KLS2}, the authors observed that after a change of coordinates, function $v_z = u^+ - \kappa(z) u^-$, for each interior point $z$, satisfies 
\begin{equation*}
\Delta v_z = \ddiv F_z,
\end{equation*}
where $\norm{F_z}_{L^2(B_r(z))} = O(r^{\frac{n}{2} + \alpha - \e})$, for any small $\e$, provided that $a_+$, $a_-$ and $f$ are all H\"older continuous with exponent $\alpha$. Hence, one can approximate $v_z$ with a harmonic polynomial by a classical iteration argument. By this approach, the Lipschitz regularity of solution $u$ as well as the $C^{1,\alpha}$-regularity of its free boundary $\{u = 0\}$ were revisited, and were also proven in a wider class than the one considered in \cite{KLS1}, since the monotonicity of the Alt-Caffarelli-Friedman functional fails when the right hand side, $\ddiv f$, of \eqref{eq:main-gen} is given with $f\in C^{0,\alpha} \setminus C^{0,1}$. Moreover, a structure theorem was established for the singular part of the free boundary, i.e., where the solution vanishes faster than any linear function, which asserts that vanishing order of the solution has to be an integer, and the set where the solution vanishes with the same (finite) order is contained in a countable union of $C^1$-manifolds. 

In this paper, we extend from H\"older to Dini regime the main assertions in \cite{KLS2}, which are the Lipschitz regularity (Theorem \ref{theorem:int-Lip}) of solutions, continuous differentiability (Theorem \ref{theorem:fb-reg}) of the nodal set around regular points, and the structure theorem (Theorem \ref{theorem:sing}) of the singular set; in addition to that, we also establish uniform doubling property (Theorem \ref{theorem:double}) and a geometric measure estimate (Theorem \ref{theorem:meas}) of the nodal set. 

One can also study \eqref{eq:main-gen} from the viewpoint of the nodal set theory. Due to the rich amount of literature in the area, we shall only discuss some important results closely related to our work. We refer to \cite{Bers}, \cite{Chen} and \cite{Han}, for some development in early stages regarding (higher-order) approximation of solutions by (homogeneous) harmonic polynomials around their vanishing points. As mentioned above, most of the literature in this direction is concerned with H\"older coefficients, and we believe that this paper is the first to bring the theory down to the regime of Dini coefficients. 

Concerning the measure estimate of the nodal set, we would like to mention classical results \cite{HS} and \cite{HL94}, which study the nodal set of solutions to uniformly elliptic PDEs in non-divergence form with merely continuous coefficients (so the argument can be easily modified for H\"older continuous PDEs in divergence form). Most of the literature, however, considers coefficients with Lipschitz regularity, since then Almgren's frequency formula becomes monotone \cite{GL}, from which one can deduce fine properties of the nodal set. 

Recently, there has been some outstanding improvement in \cite{CNV} and \cite{NV} regarding the measure estimate of singular/critical sets, where properties of harmonic functions and solutions to elliptic PDEs with Lipschitz coefficients are thoroughly investigated via the monotonicity of Almgren's frequency. The latter result, in particular, significantly improves the classical estimate, e.g., \cite{HHL}, which was only available for smooth PDEs. Moreover, as a byproduct, it also provides some valuable estimates for the size of (so-called effective) nodal sets. 

In connection to our structure theorem (Theorem \ref{theorem:sing}), we would also like to address a recent article \cite{BET} that studies the structure of a set that can be approximated by the nodal sets of harmonic polynomials. We believe that our nodal/singular set can also be analysed by their approach, as our solution after adjusting the break (i.e., $u^+ - \kappa(z) u^-$) can be approximated by homogeneous harmonic polynomials (Lemma \ref{lemma:int-C1} and Lemma \ref{lemma:int-Cd}) at each vanishing point. However, our result is (strictly) finer, since it gives the regularity of the manifolds that cover the lower dimensional singular sets as well. 

This paper is organised as follows. In Section \ref{section:notation}, we introduce some notation, terminologies and standing assumptions that will be used throughout the paper. In Section \ref{section:int-reg}, we prove Lipschitz regularity of solutions (Theorem \ref{theorem:int-Lip}). In Section \ref{section:reg}, we establish continuous differentiability of the regular part of nodal sets (Theorem \ref{theorem:fb-reg}). In Section \ref{section:sing}, we present a structure theorem for singular sets (Theorem \ref{theorem:sing}). In Section \ref{section:double}, we study uniform doubling property of solutions vanishing with finite orders (Theorem \ref{theorem:double}), and finally in Section \ref{section:nodal} we estimate measure of the nodal sets.  


\section{Notation and Standing Assumption}\label{section:notation}

By $n$ we denote the space dimension, and we shall always assume $n\geq 2$. By $\Omega$ we denote a domain, i.e., open, connected set, in $\R^n$. By $\N$ we denote the set of natural numbers, $\{1,2,3\cdots\}$. Given $z\in\R^n$ and $r>0$, by $B_r(z)$ we denote the ball in $\R^n$ of radius $r$ centred at $z$. Given a set $E$, by $\overline E$ we denote the closure of $E$, and by $\partial E$ its topological boundary, $\overline E\setminus E$. By $H^{n-1}$ we denote the $(n-1)$-dimensional Hausdorff measure. 

We shall say that $\omega$ is a modulus of continuity, if $\omega:(0,1]\ra (0,\infty)$ satisfies $\lim_{r\ra 0} \omega(r) = 0$, $\omega(1) < \infty$, $\omega(r)$ is nondecreasing while $r^{-1}\omega(r)$ is non-increasing in $r$. We shall call $\omega$ a Dini modulus of continuity, or say that $\omega$ satisfies the Dini condition, if $\omega$ is a modulus of continuity satisfying $
\int_0^1 r^{-1}\omega(r) \,dr < \infty$. Also given a Dini modulus of continuity $\omega$, we shall denote by $\omega_1$ the modulus of continuity defined by $\omega_1 (r) = \int_0^r s^{-1}\omega(s)\,ds + r\int_r^1 s^{-2}\omega(s)\,ds$, unless stated otherwise. 

As usual, space $C^k(\Omega)$ consists of all $k$-times continuously differentiable functions on $\Omega$. Given a modulus of continuity $\omega$, $C^{k,\omega}(\Omega)$ consists of all functions $f\in C^k(\Omega)$, for which the norm 
\begin{equation*}
\norm{f}_{C^{k,\omega}(\Omega)} = \max_{|\beta|\leq k} \sup_{x\in \Omega} |\partial^\beta f(x)| + \max_{|\beta| = k} \sup_{\substack{x,y\in\Omega \\ 0<|x-y|\leq 1}} \frac{|\partial^\beta f(x) - \partial^\beta f(y)|}{\omega(|x-y|)}
\end{equation*} 
is finite. Space $C_{loc}^{k,\omega}(\Omega)$ consists of all functions $f\in C^{k,\omega}(K)$ for every compact subset $K\subset\Omega$. By $C^{k,\alpha}(\Omega)$ (and similarly $C_{loc}^{k,\alpha}(\Omega)$) we denote the space $C^{k,\omega}(\Omega)$ (and respectively $C_{loc}^{k,\omega}(\Omega)$) with $\omega(r) = cr^\alpha$ for some $c>0$ and $\alpha\in(0,1]$. 

Class $W^{k,p}(\Omega)$ is the usual Sobolev space consisting of all $k$-times weakly differentiable functions on $\Omega$ whose derivatives of any order less than or equal to $k$ belong to $L^p(\Omega)$. The local version $W_{loc}^{k,p}(\Omega)$ is defined accordingly. By $W^{-1,2}(\Omega)$ we denote the dual space of $W_0^{1,2}(\Omega)$.

Given a (Sobolev) function $u$, by the nodal set of $f$ we indicate its zero-level set, $\{ u = 0 \}$. The notion of regular and singular part of the nodal set is given in Section \ref{section:reg} (see \eqref{eq:reg}) and respectively Section \ref{section:sing} (see \eqref{eq:sing}). 

Some of the basic elliptic regularity theory will be used without reference, e.g., local energy estimate (or Cacciopoli's inequality), local $L^\infty$-estimate, local $W^{1,p}$-estimate for PDEs with continuous coefficients, etc., and we refer to \cite{GT} for details. 

We shall intentionally distinguish the term ``function'' from ``mapping'', where the former will always be real-valued, while the latter can have values in different spaces, such as vectors, matrices or tensors. We shall also follow the summation convention for repeating indices, and the derivatives with multi-indices.

Given two square matrices $A, B \in \R^{n^2}$, by saying $A\leq B$ we shall indicate that eigenvalues of $A-B$ are all nonnegative. By $I$ we shall denote the identity matrix in $\R^{n^2}$, unless stated otherwise. By $e_i$, with $i\in\{1,\cdots,n\}$, we denote the standard $i$-th basis vector in $\R^n$. Given a square matrix $A$, we call $A$ elliptic, if $A \geq \lambda I$ for some positive constant $\lambda$. Given a function $v$ and a square matrix $A$, we call $v$ $A$-harmonic, if $\ddiv (A \nabla v) = 0$.

Let us list up the standing assumptions on coefficient matrices $a_+,a_-: \Omega\ra \R^{n^2}$, which will always be real-symmetric $(n\times n)$-dimensional matrix-valued mappings, as follows: 
\begin{enumerate}[(i)]
\item (Ellipticity) There is a constant $\lambda\in(0,1)$ such that
\begin{equation}\label{eq:a-ellip}
\lambda I \leq \inf_\Omega a_\pm \leq \sup_\Omega a_\pm \leq \frac{1}{\lambda} I.
\end{equation}
\item (Regularity) There is a Dini modulus of continuity $\omega$ such that $a_+,a_- \in (C^{0,\omega}(\Omega))^{n^2}$, and 
\begin{equation}\label{eq:a-Dini}
\sup_{\substack{x,y\in\Omega \\ 0<|x-y|\leq 1}} \frac{|a_\pm (x) - a_\pm (y) |}{\omega(|x-y|)} \leq 1. 
\end{equation} 
\item (Structure) There is a real-valued function $\kappa:\Omega\ra (0,\infty)$ such that 
\begin{equation}\label{eq:a-k}
a_- = \kappa a_+ \quad\text{on }\Omega. 
\end{equation} 
\end{enumerate}
Note that \eqref{eq:a-ellip} implies in \eqref{eq:a-k} that $\lambda^2 \leq \kappa \leq \lambda^{-2}$ on $\Omega$. 
Parameters $n$, $\lambda$ and $\omega$ will be fixed throughout the paper, and will often be called universal.

Given $f\in (L^2(\Omega))^n$, we shall say that $u\in W^{1,2}(\Omega)$ is a weak solution (or simply a solution) to \eqref{eq:main-gen}, if
\begin{equation*}
\int_\Omega (a_+ \nabla u^+ - a_- \nabla u^-)\cdot \nabla \phi \,dx = \int_\Omega f\cdot \nabla \phi\,dx,
\end{equation*} 
for any $\phi\in W_0^{1,2}(\Omega)$.


\section{Interior Regularity of Solution}\label{section:int-reg}

This section is devoted to interior regularity of weak solutions to \eqref{eq:main-gen}, provided that $a_+$, $a_-$ and $f$ are Dini continuous. Our main objective is to prove the following. 

\begin{theorem}\label{theorem:int-Lip}
Let $a_+,a_- \in (C^{0,\omega}(B_1))^{n^2}$ satisfy \eqref{eq:a-ellip}, \eqref{eq:a-Dini} and \eqref{eq:a-k}, and let $f\in (C^{0,\omega}(B_1))^n$ be given. Suppose that $u\in W^{1,2}(\Omega)$ is a weak solution to \eqref{eq:main-gen}. Then $u \in W^{1,\infty}(B_{1/2})$, and there is a unique mapping $Vu : B_{1/2}\ra \R^n$, defined as in \eqref{eq:Vu}, such that for any (finite) $p>n$, 
\begin{equation}\label{eq:int-C1-re}
\begin{split}
&\sup_{z\in \overline{B_{1/2}}}\sup_{ r \in (0,\frac{1}{4})} \left( \frac{1}{r^n\omega_1(r)^p} \int_{B_r(z)} |\nabla u^+ - \kappa(z) \nabla u^- - Vu(z)|^p\,dx\right)^{\frac{1}{p}}\\
&\leq C_p \left( \norm{u}_{L^2(B_1)} + \norm{f}_{C^{0,\omega}(B_1)}\right),
\end{split}
\end{equation}
with a positive constant $C_p$ determined by $n$, $\lambda$, $\omega$ and $p$ only, as well as a modulus of continuity $\omega_1$ determined solely by $\omega$. In particular, one has  
\begin{equation}\label{eq:int-Lip}
\norm{u}_{W^{1,\infty}(B_{1/2})} + [Vu]_{C^{0,\omega_1}(B_{1/2})} \leq C \left( \norm{u}_{L^2(B_1)} + \norm{f}_{C^{0,\omega}(B_1)}\right),
\end{equation} 
where $C$ is a positive constant depending only on $n$, $\lambda$ and $\omega$.
\end{theorem}

This theorem will follow from a point-wise $C^1$-approximation of $v_z = u^+ - \kappa(z) u^-$, that is uniform for each interior point $z\in\Omega$, where $\kappa(z)$ is the real number for which we have $a_-(z) = \kappa(z) a_+(z)$. As in the collaboration \cite{KLS2} with Lee, Shahgholian and the author, here we shall also transform our PDE \eqref{eq:main-gen}, into
\begin{equation*}
\Delta v_z = \ddiv F_z,
\end{equation*}
for some $F_z$ satisfying
\begin{equation*}
\norm{F_z}_{L^2(B_r(z))} = O(1) \omega(r)\left( \norm{\nabla u}_{L^2(B_r(z))} + r^{\frac{n}{2}}\right). 
\end{equation*} 
However, the argument presented in \cite{KLS2} is restricted to the case of H\"older modulus of continuity $\omega$, since it estimates $\norm{\nabla u}_{L^2(B_r(z))}$ prior to the $C^1$-approximation, and as a result it obtains $\norm{\nabla u}_{L^2(B_r(z))} = O(r^{\frac{n}{2} - \e})$, for all small $\e>0$, only. Such an estimate is sufficient to work with a H\"older modulus of continuity, say $C^\alpha$, because it still gives us $\norm{F_z}_{L^2(B_r(z))} = O(r^{\frac{n}{2} + \alpha - \e})$. Now if we choose $\e$ sufficiently small such that $\e < \alpha$, then an iteration argument will lead us to $\norm{v_z - l_z}_{L^2(B_r(z))} = O(r^{\frac{n}{2} + 1 + \alpha - \e})$, as desired. 

Nevertheless, this argument would fail if $\omega$ is only Dini continuous, yet not H\"older, since Dini modulus of continuity allows $r^\e = o(\omega(r))$, for any $\e>0$. Here we fill this gap, by employing a new approximation lemma adopted to our problem \eqref{eq:main-gen}. 

\begin{proposition}\label{proposition:v-C1}
Given a Dini modulus of continuity $\omega$, there exist constants $\rho\in(0,\frac{1}{4})$ and $C>0$, depending only on $n$ and $\omega$, such that if $v\in W^{1,2}(B_1)$ satisfies
\begin{equation}\label{eq:v-pde}
\max\left\{ \norm{v}_{W^{1,2}(B_1)}, \sup_{r\in(0,1)} \frac{\norm{\Delta v}_{W^{-1,2}(B_r)} }{\omega(\rho r) ( \norm{\nabla v}_{L^2(B_r)}+ r^{\frac{n}{2}}) } \right\}  \leq 1,
\end{equation} 
then there is an affine function $l$ such that 
\begin{equation}\label{eq:v-C1}
|l(0)| + |\nabla l| + \sup_{r\in(0,1)} \frac{\norm{v - l}_{L^2(B_r)} + r\norm{\nabla (v-l)}_{L^2(B_r)}}{r^{\frac{n}{2} + 1}\omega_1(r)}\leq C, 
\end{equation} 
where $\omega_1$ is a modulus of continuity determined by $\omega$ alone. 
\end{proposition}

\begin{remark}\label{remark:v-C1}
By a more thorough analysis, one can prove that $\omega_1$ is given by 
\begin{equation*}
\omega_1 (r) = \int_0^r\frac{\omega(s)}{s}\,ds + r\int_r^1 \frac{\omega(s)}{s^2}\,ds.
\end{equation*}
Hence, the Dini condition on $\omega$ is essential to have $\omega_1$ well-defined and $\lim_{r\ra 0 }\omega_1(r) = 0$. 
\end{remark}

\begin{proof}[Proof of Proposition \ref{proposition:v-C1}]
Let $\bar{c}>1$ be a constant to be determined later, by $n$ and $\omega$. Choose $\mu\in(0,\frac{1}{2})$ by a sufficiently small real satisfying 
\begin{equation}\label{eq:mu}
4\bar{c} \omega(1) \mu \leq \omega(\mu),
\end{equation}
and accordingly select a small real $\rho\in(0,\frac{1}{2})$ for which 
\begin{equation}\label{eq:rho}
4\bar{c} \omega(1) \psi(\rho) \leq \mu^{\frac{n}{2} + 1}\omega(\mu)\quad\text{with}\quad \psi(\rho) = \sum_{k=0}^\infty \omega(\rho \mu^k).
\end{equation} 
The inequality \eqref{eq:rho} is valid for any small $\rho$, once $\mu$ is fixed, due to the Dini condition on $\omega$. In what follows, let us denote by $\beta_n$ the volume of the unit ball. 

We are going to construct a sequence $\{l_k\}_{k=1}^\infty$ of affine functions that for each $k\in\N\cup\{0\}$, 
\begin{equation}\label{eq:v-lk-L2}
 \frac{\norm{v - l_k}_{L^2(B_{\mu^k})} }{\mu^k}  +  \norm{\nabla (v - l_k)}_{L^2(B_{\mu^k})} \leq \frac{\mu^{\frac{kn}{2}}\omega(\rho \mu^k)}{\omega(\rho)}, \quad\text{and}
\end{equation}
\begin{equation}\label{eq:lk}
 \frac{ |l_k (0) - l_{k-1} (0) |}{\mu^k} +  |\nabla l_k - \nabla l_{k-1} | \leq c_0\frac{\omega(\rho\mu^k)}{\omega(\rho)},
\end{equation}  
with $l_{-1} = 0$, where $c_0>0$ depends only on $n$.

Set $l_0 = 0$. Then \eqref{eq:lk} is trivial for $k=0$, and \eqref{eq:v-lk-L2} for $k=0$ also follows immediately from the assumption $\norm{v}_{W^{1,2}(B_1)}\leq 1$. 

Let us fix $m\in\N\cup\{0\}$, and suppose that induction hypotheses \eqref{eq:v-lk-L2} and \eqref{eq:lk} are satisfied with certain affine function $l_k$ for each $0\leq k\leq m$.

Define $v_m :B_1\ra \R$ by  
\begin{equation}\label{eq:vm-um}
v_m (x) = \frac{\omega(\rho)}{\mu^m \omega(\rho \mu^m)} (v - l_m)(\mu^m x),
\end{equation}
so that $v_m \in W^{1,2}(B_1)$ and $\norm{v_m}_{W^{1,2}(B_1)} \leq 1$, due to the induction hypothesis \eqref{eq:v-lk-L2} for $k=m$. Moreover, since $\nabla l_m$ is a constant vector, we deduce from \eqref{eq:v-pde} that 
\begin{equation}\label{eq:vm-pde}
\begin{split}
\norm{\Delta v_m}_{W^{-1,2}(B_1)} &= \frac{ \omega(\rho)}{\mu^{mn/2}\omega(\rho\mu^m)} \norm{\Delta v}_{W^{-1,2}(B_{\mu^m})} \\
&\leq \omega(\rho)\left( \frac{ \norm{\nabla v}_{L^2(B_{\mu^m})}}{\mu^{mn/2}} + 1\right) \\
&\leq \omega(\rho\mu^m) \norm{\nabla v_m}_{L^2(B_1)} + \omega(\rho) \left( \norm{\nabla l_m}_{L^2(B_{\mu^m})} + 1  \right)\\
&\leq \left( c_0\sqrt{\beta_n} + 2 \right)\psi(\rho),
\end{split} 
\end{equation} 
where in the derivation of the last inequality we used $\norm{\nabla v_m}_{L^2(B_1)} \leq 1$, $l_0 = 0$ and \eqref{eq:lk} for all $k\leq m$, which yields 
\begin{equation}
|\nabla l_m|\leq |\nabla l_0| + \sum_{k=1}^m |\nabla l_k - \nabla l_{k-1}|\leq \frac{c_0}{\omega(\rho)} \sum_{k=1}^\infty\omega(\rho\mu^k) \leq c_0\frac{\psi(\rho)}{\omega(\rho)},
\end{equation} 
as well as $\omega(\rho) < \psi(\rho)$. 

Now let $h\in W^{1,2}(B_1)$ be the harmonic function on $B_1$ subject to boundary condition $h - v_m \in W_0^{1,2}(B_1)$. Then it follows from $\int_{B_1} \nabla h\cdot \nabla (v_m - h)\,dx = 0$ and \eqref{eq:vm-pde} that 
\begin{equation}\label{eq:Dvm-Dh-L2}
\norm{\nabla (v_m - h)}_{L^2(B_1)} \leq \norm{\Delta v_m }_{W^{-1,2}(B_1)} \leq \left(c_0 \sqrt{\beta_n} + 2 \right) \psi(\rho). 
\end{equation} 
Thus, the Poincar\'e inequality implies
\begin{equation}\label{eq:vm-h-L2}
\norm{v_m - h}_{L^2(B_1)} \leq c_1\left(c_0 \sqrt{\beta_n} + 2 \right) \psi(\rho),
\end{equation}  
where $c_1>0$ depends only on $n$. 

On the other hand, as $h$ being a harmonic function on $B_1$, we deduce from \eqref{eq:vm-h-L2} and \eqref{eq:v-lk-L2} (which yields $\norm{v_m}_{L^2(B_1)} \leq 1$) that 
\begin{equation}\label{eq:h-W2inf}
\norm{h}_{W^{2,\infty}(B_{1/2})}  \leq c_2 \norm{h}_{L^2(B_1)} \leq 2c_2,
\end{equation}
where we used the smallness condition \eqref{eq:rho} for $\rho$ that gives $c_1(c_0\beta_n + 2)\psi(\rho)\leq 1$ (as we will choose $\bar{c}$ later to be larger than $c_1(c_0\sqrt{\beta_n} + 2)$). Here $c_2>0$ is also a constant depending only on $n$. Thus, denoting by $q_m$ the affine part of $h$, i.e., 
\begin{equation*}
q_m (x) = h(0) + \nabla h(0)\cdot x,
\end{equation*}
we can deduce from \eqref{eq:h-W2inf} and the Taylor expansion that with $\mu\in(0,\frac{1}{2})$, 
\begin{equation}\label{eq:h-qm}
 \norm{ h - q_m }_{L^\infty(B_\mu)} + \mu \norm{\nabla (h-q_m)}_{L^\infty(B_\mu)}  \leq c_3 \mu^2,
\end{equation}
for certain constant $c_3>0$ depending only on $n$. Thus, by \eqref{eq:mu}, \eqref{eq:rho} with $\bar{c} =\max\{ c_1(c_0\sqrt{\beta_n} + 2), c_3\sqrt{\beta_n}\}$, \eqref{eq:vm-h-L2} and \eqref{eq:h-qm}, we observe that
\begin{equation}\label{eq:vm-qm-L2}
\norm{v_m-q_m}_{L^2(B_\mu)} \leq \bar{c}\left( \psi(\rho) + \mu^{\frac{n}{2} + 2} \right) \leq \frac{\mu^{\frac{n}{2} + 1} \omega(\mu) }{2\omega(1)},
\end{equation} 
and similarly, by \eqref{eq:Dvm-Dh-L2} instead of \eqref{eq:vm-h-L2}, 
\begin{equation}\label{eq:Dvm-Dqm-L2}
\norm{\nabla (v_m - q_m)}_{L^2(B_\mu)} \leq \frac{\mu^{\frac{n}{2}}\omega(\mu)}{2\omega(1)}. 
\end{equation} 

To this end, let us define 
\begin{equation}\label{eq:lm}
l_{m+1} = l_m + \frac{\mu^m \omega(\rho \mu^m)}{\omega(\rho)} q_m \left( \frac{\cdot}{\mu^m} \right).
\end{equation}
Since $\omega$ is a Dini modulus of continuity, $\omega(1)\omega(\rho r) \geq \omega(\rho) \omega(r)$ for any $r\in(0,1)$. Thus, we can verify from \eqref{eq:h-W2inf} the last induction hypothesis \eqref{eq:lk} for $k = m+1$, with $c_0 = 2c_2$, and the first hypothesis \eqref{eq:vm-qm-L2} for $k = m+1$ from \eqref{eq:vm-qm-L2} and \eqref{eq:Dvm-Dqm-L2}. This finishes the proof.
\end{proof} 

As a corollary, we obtain a uniform point-wise $C^1$-approximation for $u^+ - \kappa(z) u^-$ at each interior point $z$, when $u$ is a weak solution to \eqref{eq:main-gen}. Here we obtain an estimate in a stronger space, i.e., $L^\infty$ for the function and $L^p$ for the gradient, due to the Dini continuity of our source term, $f$.

\begin{lemma}\label{lemma:int-C1}
Under the setting of Theorem \ref{theorem:int-Lip}, for every $z\in \overline{B_{1/2}}$, there exists a (unique) affine function $l_z$ such that for any (finite) $p>n$, 
\begin{equation}\label{eq:phiz-lz-W12}
\begin{split}
& |l_z(z)| + |\nabla l_z| + \sup_{r\in (0,\frac{1}{4})} \frac{\norm{v_z - l_z}_{L^\infty(B_r(z))}}{r\omega_1 (r)} \\
&\quad + \sup_{r\in (0,\frac{1}{4})}\frac{\norm{\nabla (v_z - l_z)}_{L^p(B_r(z))}}{r^{\frac{n}{p}}\omega_1(r)} \\
& \leq C_p\left( \norm{u}_{L^2(B_1)} + \norm{f}_{C^{0,\omega}(B_1)} \right),
\end{split}
\end{equation} 
where $v_z = u^+ - \kappa(z) u^-$, $C_p$ is a positive constant depending only on $n$, $\lambda$, $\omega$ and $p$, and $\omega_1$ is a modulus of continuity determined by $\omega$ only. 
\end{lemma} 

\begin{remark}\label{remark:int-C1}
Unlike in \cite{KLS2}, we obtain this estimate for every $z\in\overline{B_{1/2}}$, regardless of whether $z\in \{u=0\}$ or not. This observation will be used significantly later in Section \ref{section:double} and Section \ref{section:nodal}. 
\end{remark}

\begin{proof}
Given an arbitrary point $z\in \overline{B_{3/4}}$, function $v_z = u^+ - \kappa(z) u^-$ satisfies
\begin{equation}\label{eq:vz-pde}
\ddiv (a_+(z) \nabla v_z) = \ddiv F_z,
\end{equation} 
in the weak sense, with $F_z: B_1\ra\R^n$ defined by 
\begin{equation}\label{eq:Fz}
F_z = (a_+(z) - a_+)\nabla u^+ - (a_-(z) - a_-) \nabla u^- + (f- f(z)). 
\end{equation} 
Thus, denoting by $M$ the $C^{0,\omega}$-norm of $f$, we can deduce from \eqref{eq:a-Dini} that 
\begin{equation*}
\begin{split} 
\norm{F_z}_{L^2(B_r(z))} &\leq (1 + M)\omega(r) \left(  \norm{ \nabla u}_{L^2(B_r(z))} + r^{\frac{n}{2}} \right) \\
& \leq \lambda^{-2} ( 1 + M) \omega(r) \left(  \norm{ \nabla v_z}_{L^2(B_r(z))} + r^{\frac{n}{2}} \right), 
\end{split} 
\end{equation*} 
for any $r\in (0,\frac{1}{4})$, where in the derivation of the second inequality we used $\norm{\nabla u}_{L^2(B_r(z))} \leq \max\{1,\kappa(z)^{-1}\}\norm{\nabla v_z}_{L^2(B_r(z))}$ and $\lambda^2 \leq \kappa(z) \leq \lambda^{-2}$, which can be deduced from \eqref{eq:a-ellip} and \eqref{eq:a-k}. Hence, a change of coordinate system (so that $a_+(z)$ is transformed to $I$) followed by a suitable scaling argument allows us to apply Proposition \ref{proposition:v-C1} and deduce that there exists a unique affine function, $l_z$, for which
\begin{equation}\label{eq:phiz-lz-W12-re}
\begin{split}
& |l_z(z)| + |\nabla l_z| + \sup_{r\in (0,\frac{1}{8})} \frac{\norm{v_z - l_z}_{L^2(B_r(z))} +r \norm{\nabla (v_z - l_z)}_{L^2(B_r(z))} }{r^{\frac{n}{2} + 1} \omega_1 (r)} \\
& \leq C_0\left( \norm{u}_{L^2(B_1)} + \norm{f}_{C^{0,\omega}(B_1)}\right),
\end{split}
\end{equation}
where $C_0>0$ depends only on $n$, $\lambda$ and $\omega$. We remark that hypothesis $\norm{v}_{W^{1,2}(B_1)} \leq 1$ in the proposition can be verified from the local energy estimate of $u$, hence of $v_z$, due to the uniform ellipticity of the operator, $a_+\chi_{\{u>0\}} + a_- \chi_{\{u \leq 0\}}$, of \eqref{eq:main-gen}. 

Note that the above estimate is in $L^2$-norm. In order to improve this norm, we need a bootstrap argument, and for this purpose let us prove that $u\in W^{1,\infty}(B_{3/4})$ (hence, the first part of Theorem \ref{theorem:int-Lip} is established). Combining \eqref{eq:phiz-lz-W12-re} with inequality $\lambda^2 \leq \kappa(z) \leq \lambda^{-2}$, we observe that
\begin{equation*}
\sup_{r\in(0,\frac{1}{4})} \left(\frac{1}{r^n} \int_{B_r(z)} (u^2 + |\nabla u|^2)\,dx\right)^{\frac{1}{2}}\leq C_1\left( \norm{u}_{L^2(B_1)} + \norm{f}_{C^{0,\omega}(B_1)}\right),
\end{equation*}
where $C_1>0$ depends on $n$, $\lambda$ and $\omega$ only. Taking $r\ra 0$ yields, and assuming that $z$ is a Lebesgue point of both $u$ and $\nabla u$, we obtain  
\begin{equation}\label{eq:u-Du-Linf}
|u(z)| + |\nabla u(z)| \leq C_1( 1 + M).
\end{equation}
Since $u\in W^{1,2}(B_1)$, we conclude that \eqref{eq:u-Du-Linf} holds a.e. in $B_{3/4}$. 

Now we can go back to the PDE \eqref{eq:vz-pde} for $v_z$ and observe that
\begin{equation*}
\ddiv (a_+ (z) \nabla (v_z - l_z)) = \ddiv F_z,
\end{equation*} 
with $F_z$ being the same as in \eqref{eq:Fz}. Now that $u\in W^{1,\infty}(B_{3/4})$, we have $F_z \in (L^\infty(B_{3/4}))^n$. Therefore, for every $z\in \overline{B_{1/2}}$, we can adopt the interior $L^\infty$-estimate and $W^{1,p}$-estimate, for any $p>n$, from the classical regularity theory and deduce that for any $r\in(0,\frac{1}{8})$, 
\begin{equation*}
\begin{split}
& \norm{v_z - l_z}_{L^\infty(B_r(z))} + r^{1-\frac{n}{p}} \norm{\nabla (v_z - l_z)}_{L^p(B_r(z))}\\
& \leq C_p \left( r^{-\frac{n}{2}} \norm{v_z - l_z}_{L^2(B_{2r}(z))} + r \norm{F_z}_{L^\infty(B_{2r}(z))} \right) \\
& \leq \tilde{C}_pr\omega_1(r) \left( \norm{u}_{L^2(B_1)} + \norm{f}_{C^{0,\omega}(B_1)}\right),
\end{split}
\end{equation*}
where $C_p$ and $\tilde{C}_p$ depend at most on $n$, $\lambda$, $\omega$ and $p$. This finishes the proof of \eqref{eq:phiz-lz-W12}. 
\end{proof} 

Note that the affine mapping, $l_z$, above is uniquely determined for every $z$, and hence $z\mapsto \nabla l_z$ becomes a well-defined vector-valued mapping. The next lemma asserts that this mapping is uniformly continuous. 

\begin{lemma}\label{lemma:int-C1-more}
Under the setting of Lemma \ref{lemma:int-C1}, mapping $z\mapsto \nabla l_z$ belongs to class $(C^{0,\omega_1}(B_{1/2}))^n$, and for any pair $(z,z')$ of distinct points in $B_{1/2}$, 
\begin{equation}\label{eq:Dlz-C}
\frac{|\nabla l_z - \nabla l_{z'}|}{\omega_1(|z-z'|)} \leq C\left( \norm{u}_{L^2(B_1)} + \norm{f}_{C^{0,\omega}(B_1)} \right).
\end{equation}  
\end{lemma}

\begin{proof} 
Denote by $w_i$ and $q_i$ function $v_{z_i} = u^+ - \kappa(z_i) u^-$ and respectively vector $\nabla l_{z_i}$. It follows immediately from \eqref{eq:a-ellip}, \eqref{eq:a-Dini} and \eqref{eq:a-k} that $|\kappa(z_1) - \kappa(z_2)| \leq C_1\omega(|z_1 - z_2|)$, where $C_1$ depends only on $n$, $\lambda$ and $\omega$. Therefore, noting that $\nabla w_1 - \nabla w_2 = (\kappa(z_1) - \kappa(z_2)) \nabla u^-$ a.e., we can deduce from the local $L^\infty$-estimate \eqref{eq:u-Du-Linf} of $|\nabla u|$ that
\begin{equation}\label{eq:Dv1-Dv2-L2}
\norm{\nabla w_1 - \nabla w_2 }_{L^\infty(B_{3/4})} \leq C_2 \omega(|z_1 - z_2|),
\end{equation}
for another constant $C_2>0$ depending only on $n$, $\lambda$ and $\omega$. 

Now it suffices to consider the case where $|z_1 - z_2| < \frac{1}{8}$, as the general case will follow from a covering argument. Choose $r_0\in(0,\frac{1}{8})$ such that $r_0 < |z_1 - z_2 | \leq 2r_0$. Then from \eqref{eq:phiz-lz-W12-re} and \eqref{eq:Dv1-Dv2-L2}, we obtain  
\begin{equation*}
\begin{split}
|q_1 - q_2| &\leq \beta_n^{-\frac{1}{2}}r_0^{-\frac{n}{2}} \left( \sum_{i=1}^2 \norm{\nabla w_i - q_i}_{L^2(B_{r_0}(z_1))} + \norm{\nabla (w_1 -  w_2)}_{L^2(B_{r_0}(z_1))}\right) \\
&\leq \beta_n^{-\frac{1}{2}} r_0^{-\frac{n}{2}} \sum_{i=1}^2 \norm{\nabla w_i - q_i}_{L^2(B_{2r_0}(z_i))} + C_2 \omega(|z_1- z_2|) \\
&\leq C_3 (\omega_1(2r_0) + \omega(|z_1 - z_2|) ) \\
& \leq C_4 \omega_1(|z_1 - z_2|),
\end{split}
\end{equation*} 
where $\beta_n$ is the volume of the $n$-dimensional unit ball, while $C_3$ and $C_4$ are positive constants depending at most on $n$, $\lambda$ and $\omega$. This finishes the proof. 
\end{proof} 

Theorem \ref{theorem:int-Lip} is now a simple byproduct of Lemma \ref{lemma:int-C1} and Lemma \ref{lemma:int-C1-more}.

\begin{proof}[Proof of Theorem \ref{theorem:int-Lip}]
Recall from the proof of Lemma \ref{lemma:int-C1} that \eqref{eq:u-Du-Linf} holds a.e. in $\Omega$, hence the local $W^{1,\infty}$-estimate \eqref{eq:int-Lip} of $u$ is established. Next, we set $Vu: B_{1/2}\ra \R^n$ by 
\begin{equation}\label{eq:Vu}
Vu(z) =\nabla l_z,
\end{equation}
where $l_z$ is the unique affine function satisfying \eqref{eq:phiz-lz-W12}. Then \eqref{eq:int-C1-re} and \eqref{eq:int-Lip} are simply an alternative expression of \eqref{eq:phiz-lz-W12} and respectively \eqref{eq:Dlz-C}. 
\end{proof} 


\section{Regular Part of Nodal Set}\label{section:reg} 

This section is concerned with the regularity of the nodal set, i.e., the zero-level set, of a weak solution to \eqref{eq:main-gen}. We are going to study the regularity of $\{u=0\}$ around regular (or non-degenerate) points, which are the points where its ``gradient'' does not vanish, while the solutions itself vanishes. However, as the gradient of $u$ cannot be defined on the nodal set, we shall define these points by its growth order. More precisely, we shall define the class of regular points by
\begin{equation}\label{eq:reg}
\begin{split}
\cN(u) &= \{u=0\}\cap \left\{ z \in B_1: \limsup_{r\ra 0} \frac{\norm{u}_{L^2(B_r(z))}}{r^{\frac{n}{2} + 1}} > 0\right\} \\
&= \{u=0\}\cap \{|Vu| > 0\},
\end{split}
\end{equation} 
where the second identity is due to Theorem \ref{theorem:int-Lip}, with $Vu$ as in the statement; note that one can actually define $Vu$ on the whole $B_1$ via a simple scaling argument. 

With the uniform point-wise $C^{1,\omega_1}$-approximation derived in Lemma \ref{lemma:int-C1}, the regularity of our nodal set around a regular point follows easily from two fundamental tools from multivariable calculus, which are Whitney's extension theorem and the inverse function theorem.

\begin{theorem}\label{theorem:fb-reg}
Under the same setting as in Theorem \ref{theorem:int-Lip}, and with $\cN(u)$ defined as in \eqref{eq:reg}, $\{u=0\}$ is closed in $B_1$, and $\cN(u)$ is relatively open in $\{u=0\}$. Moreover, for each $z\in\cN(u)$, there is a constant $r\in(0,\dist(z,\partial\Omega))$, determined by $n$, $\lambda$, $\omega$ and the value of $\limsup_{\rho\ra 0} \rho^{-\frac{n}{2} - 1}\norm{u}_{L^2(B_\rho(z))}$ only, such that $\{u=0\}\cap B_r(z)$ is locally a $C^{1,\omega_1}$-graph.
\end{theorem} 

\begin{remark}\label{remark:fb-reg}
This result generalises both of the previous works \cite{KLS1,KLS2}, as it now applies to any Dini modulus of continuity; moreover, it is optimal in the sense that it produces the same H\"older seminorm as in \cite{KLS1,KLS2} when our modulus of continuity is H\"older continuous, since $\omega_1$ in the statement is given as in Remark \ref{remark:v-C1}.
\end{remark}

\begin{proof}
For simplicity, let us consider the case $\norm{u}_{L^2(B_1)} \leq 1$, $\norm{f}_{C^{0,\omega}(B_1)} \leq 1$ and $z\in \cN(u)\cap \overline{B_{1/2}}$. The general case follows from an obvious scaling argument.

Due to Theorem \ref{theorem:int-Lip}, $Vu\in C^{0,\omega_1}(B_{3/4})$. In view of the second line in \eqref{eq:reg}, the continuity of $|Vu|$ implies that $\cN(u)$ is a relatively open set in $\{u=0\}$. 

On the other hand, it is clear from the (Lipschitz) continuity of $u$ (Theorem \ref{theorem:int-Lip}) that $\{u=0\}$ is a closed set. Hence, $\{u = 0\}\cap \overline{B_{1/2}}$ is also compact. Now due to the uniform $C^1$-approximation \eqref{eq:phiz-lz-W12}, we have 
\begin{equation*}
| Vu(z)\cdot (x-z) | \leq C_0|x-z|\omega_1(|x-z|),
\end{equation*}
for any $x,z\in \{u=0\}\cap \overline{B_{1/2}}$, where $C_0$ is a positive constant depending only on $n$, $\lambda$ and $\omega$. This together with the uniform continuity \eqref{eq:int-Lip} of $Vu$ on $\overline{B_{1/2}}$ allows us to use Whitney's extension theorem \cite{whit} (see also \cite{gla} and \cite{Fef09} that gives an explicit dependence on the modulus of continuity $\omega_1$) to assure the existence of a function $g\in C^{1,\omega_1}(\R^n)$ such that 
\begin{equation*}
g = 0,\quad \nabla g = Vu\quad\text{on }\{ u=0\}\cap\overline{B_{1/2}}\text{ and}
\end{equation*} 
\begin{equation*}
\norm{g}_{C^{1,\omega_1}(\R^n)} \leq c_0 \norm{Vu}_{C^{0,\omega_1}(\{u=0\}\cap\overline{B_{1/2}})} \leq c_0C_0,
\end{equation*}
with $c_0$ depending only on $n$ and $\omega_1$. 

Now if $z\in \cN(u)\cap \overline{B_{1/2}}$, then $|g(z)| = |Vu(z)| > 0$, so it follows from the implicit function theorem that $\{g = 0\}\cap B_r(z)$ is a $C^{1,\omega_1}$-graph, for some $r>0$ depending only on $\norm{g}_{C^{1,\omega_1}(\R^n)}$ and value $|g(z)|$, hence on $n$, $\lambda$, $\omega$ and $|Vu(z)|$ only. Since $\{u = 0 \}\cap \overline{B_{1/2}} \subset\{g=0\}$, this theorem is proved by taking $r$ smaller if necessary. 
\end{proof}


\section{Structure of Singular Set}\label{section:sing}

In this section, we are going to study the singular set of weak solutions to 
\begin{equation}\label{eq:main-sing}
\ddiv (a_+ \nabla u^+ - a_- \nabla u^-) = 0,
\end{equation}
which is the set where the solutions vanish faster than any linear function, i.e., 
\begin{equation}\label{eq:sing}
\cS(u) = \left\{ z\in \{u=0\}: \lim_{r\ra 0} \frac{\norm{u}_{L^2(B_r(z))}}{r^{\frac{n}{2} + 1}} = 0\right\}.
\end{equation}
We shall call $z\in\cS(u)$ a singular point of $u$. Our main result of this section can now be stated as follows. 

\begin{theorem}\label{theorem:sing}
Assume that $a_+,a_- \in (C^\omega(B_1))^{n^2}$ satisfy \eqref{eq:a-ellip}, \eqref{eq:a-Dini} and \eqref{eq:a-k}, and let $u\in W^{1,2}(B_1)$ be a weak solution to \eqref{eq:main-sing}. Then 
\begin{equation}\label{eq:sing-decom}
\cS(u) = \left(\bigcup_{d=2}^\infty \cS_d(u)\right) \cup \cS_\infty(u),
\end{equation}
where $\cS_d(u)$ and $\cS_\infty(u)$ are defined by \eqref{eq:sing-d} and respectively \eqref{eq:sing-inf}. Moreover, we have 
\begin{equation}\label{eq:sing-decom-d}
\cS(u)\setminus \cS_\infty(u) = \bigcup_{j=0}^{n-2} \cS^j(u),
\end{equation} 
where $\cS^j(u)$ is on a countable union of $j$-dimensional $C^{1,\omega_1}$-manifolds, with $\omega_1$ being a modulus of continuity determined solely by $\omega$, for each $j=0,1,\cdots,n-2$.
\end{theorem}

Some remarks are in order. 

\begin{remark}\label{remark:sing2}
Singular set is studied extensively for elliptic PDEs whose leading coefficients are H\"older continuous, e.g., \cite{Bers}, \cite{Han} and \cite{Chen}, where it is shown that a weak solution to such an elliptic PDE must vanish with an integer order, unless the vanishing order is infinite. Moreover, if the vanishing order is an integer, it can be approximated by a homogeneous harmonic polynomial, up to a change of coordinates, with an algebraic error estimate. 

Nevertheless, there has not yet been any analogous result when the leading coefficients are only Dini continuous, to the best of the author's knowledge. The reason that the classical argument by iteration fails with Dini continuous, yet not H\"older coefficients in a nutshell is that given a modulus of continuity $\omega$, one may have $\omega(r)^k \gg r$ for any sufficiently large $k$, unless $\omega(r) = O(r^\alpha)$ for some $\alpha\in(0,1)$ from the beginning. 
\end{remark} 

\begin{remark}\label{remark:sing}
Recall from Remark \ref{remark:v-C1} that if $\omega$ is a H\"older modulus of continuity, then so is $\omega_1$ with the same H\"older exponent, say $\alpha$. Thus, our result in Theorem \ref{theorem:sing} improves the classical one \cite{Han} in the sense that not only the singular set with the top dimension, $n-2$, is covered by $C^{1,\alpha}$-manifolds, but also are the set with all lower dimensions. 
\end{remark}

To begin with, let us introduce some terminologies regarding vanishing order. Given any positive real $d \geq 1$, define 
\begin{equation}\label{eq:sing-d}
\cS_d(u) = \left\{ z\in\{u=0\}: 0< \limsup_{r\ra 0} \frac{\norm{u}_{L^2(B_r(z))}}{r^{\frac{n}{2} + d}} < \infty \right\},
\end{equation} 
\begin{equation}\label{eq:zero-d}
\Gamma_d(u) = \left\{ z\in\{u=0\}: \lim_{r\ra 0} \frac{\norm{u}_{L^2(B_r(z))}}{r^{\frac{n}{2} + d}} = 0 \right\},\quad\text{and}
\end{equation}
\begin{equation}\label{eq:sing-inf}
\cS_\infty(u) = \bigcap_{d > 0} \Gamma_d(u).
\end{equation}
We shall say that if $z\in\cS_d(u)$, $\Gamma_d(u)$ or $\cS_\infty(u)$, then $u$ vanishes at $z$ with order exactly, at least, or respectively for any $d$. 

It is clear that $\cS(u) = \Gamma_1(u)$, $\cN(u) = \cS_1(u)$, and $\bigcup_{d\geq 1}\cS_d(u)\cup \cS_\infty(u) \subset \cS(u)$. However, the converse fails to hold in general, since $\cS_d(u)$ does not include the points where the vanishing order is not algebraic; for instance, if $u(x) = |x|^2 \log \frac{1}{|x|}$, then $0\in \cS(u)$ but $0\not\in \cS_d(u)$ for any $d \in (1,2]$. 

We shall establish, for each $d\in\N$, a point-wise $C^d$-approximation of $u^+-\kappa (z) u^-$ by a homogeneous harmonic polynomial of degree $d$, after a change of coordinate, if necessary, at any $z\in\Gamma_{d-1}(u)$, hence proving that $\Gamma_{d-1}(u) = \cS_d(u)\cup \Gamma_d(u)$. Then the first decomposition \eqref{eq:sing-decom} of the singular set, $\cS(u)$, follows by an induction on $d\in\N$. The second decomposition \eqref{eq:sing-decom-d}, or more accurately the structure theorem for $\cS_d(u)$, will follow from a uniformity of the point-wise $C^d$-approximation of $u^+ - \kappa(z) u^-$ along the vanishing points $z\in\cS_d(u)$.  

The point-wise $C^d$-approximation will be carried out by an iteration argument, which resembles the proof of Proposition \ref{proposition:v-C1}. However, since only the Dini continuity is allowed, we will require several intermediate steps. 

First, let us recall some classical approximation result for almost harmonic functions. 

\begin{lemma}\label{lemma:v-Cd}
Let $d\in\N$ be arbitrary, and suppose that $v\in W^{1,2}(B_1)$ satisfies 
\begin{equation*}
\max \left\{ \norm{v}_{L^2(B_1)} , \sup_{r\in(0,1)} \frac{\norm{\Delta v}_{W^{-1,2}(B_r)} }{r^{\frac{n}{2} + d-1}\omega(r)} \right\} \leq 1,
\end{equation*}
for some Dini modulus of continuity $\omega$. Then there exists a harmonic polynomial $P$ of degree $d$ for which
\begin{equation*}
\sup_{r\in(0,1)} \frac{\norm{v - P}_{L^2(B_r)} + r\norm{\nabla (v - P)}_{L^2(B_r)}}{r^{\frac{n}{2} + d} \omega_1(r),} \leq C_d
\end{equation*}
where $C_d>0$ depends only on $n$, $\omega$ and $d$, and $\omega_1$ is determined by $\omega$ alone. 
\end{lemma} 

We remark that this approximation holds without any assumption on the vanishing order of $v$ (at the origin), since the difference between $v$ and an approximating harmonic polynomial, say $P$, is again almost harmonic, with the same smallness condition of $\Delta v$, i.e., $\Delta (v - P) = \Delta v$. Let us omit the proof for Lemma \ref{lemma:v-Cd}, as it is by now considered to be standard.  

However, the situation becomes different, once we have variable coefficients, since $\ddiv(a \nabla (v-P)) = \ddiv((a - I)\nabla P)$ only, provided that $\ddiv(a\nabla v) = 0$, and the smallness condition of $|a-I|$ is not sufficient to make $\ddiv((a-I)\nabla P)$ small as well, unless we can deduce that $|\nabla P|$ is also small. This is precisely the place where the vanishing order of $v$ comes into play, and it becomes essential to find a homogeneous harmonic polynomial whose degree coincides with the vanishing order of $v$. However, we need certain uniformity for the vanishing behaviour of $v$, in order to find such a homogeneous harmonic polynomial. This is the main assertion of the next lemma. 

\begin{lemma}\label{lemma:apprx-hom}
Let $d\in \N$ be arbitrary, and let $\theta$ be a modulus of continuity. Then for any $\e>0$, there exists a constant $\delta\in(0,1)$ depending only on $n$, $\theta$, $d$ and $\e$, such that if $v  \in W^{1,2}(B_1)$ satisfies 
\begin{equation*}
\norm{v}_{W^{1.2}(B_1)}\leq 1,\quad \norm{\Delta v}_{W^{-1,2}(B_1)} \leq \delta,\quad\text{and}
\end{equation*} 
\begin{equation*}
\sup_{r\in(0,1)} \frac{\norm{v}_{L^2(B_r)}}{r^{\frac{n}{2} + d - 1}\theta(r)} \leq 1.
\end{equation*}
Then there is a harmonic function $h$ on $B_1$ such that 
\begin{equation*}
\max_{|\beta|\leq d-1}|\partial^\beta h(0)| = 0,\quad \norm{ h }_{W^{1.2}(B_1)} \leq c\quad\text{and}
\end{equation*}
\begin{equation*}
\norm{v-h}_{W^{1,2}(B_1)} \leq \e,
\end{equation*}
where $c$ is a positive constant depending only on $n$ and $d$. 
\end{lemma}

\begin{remark}\label{remark:apprx-hom}
Let us remark that the dependence of $\delta$ on $\theta$ cannot be dropped, as the following example shows. Consider, in the two-dimensional space, $v_j(x) = x_1x_2(- \log (j|x|))^{-1}$, for each $j\in\N$. Then $\Delta v_j = o(1)$ as $j\ra\infty$, and $v_j$ vanishes at the origin with order at least $2$ for all $j=1,2,\cdots$, but $v_j  \ra h$ uniformly in $B_{1/2}$, with $h(x) = x_1x_2$ having $|D^2 h| > 0$. 
\end{remark}

\begin{proof}[Proof of Lemma \ref{lemma:apprx-hom}]
Fix a constant $\bar c$ and a modulus of continuity $\theta$, where $\bar c$ will be chosen later. Suppose towards a contradiction that there exists some constant $\e_0>0$, for which we can find a sequence $\{v_j\}_{j=1}^\infty\subset W^{1,2}(B_1)$ of functions satisfying
\begin{equation}\label{eq:vj-apprx-false}
\norm{v_j}_{W^{1,2}(B_1)} \leq 1,\quad \norm{\Delta v_j}_{W^{-1,2}(B_1)} \leq \frac{1}{j},\quad\text{and}
\end{equation}
\begin{equation}\label{eq:vj-apprx-false2}
\sup_{r\in(0,\frac{1}{2})} \frac{\norm{v_j}_{L^2(B_r)}}{r^{\frac{n}{2} + d - 1}\theta(r)} \leq 1, 
\end{equation}
but for any harmonic function $h$ on $B_1$ verifying
\begin{equation}\label{eq:harm-1}
\max_{|\beta| \leq d-1} | \partial^\beta h(0)| = 0,\quad\text{and}\quad \norm{h}_{W^{1,2}(B_1)} \leq \bar c,
\end{equation}
we must have 
\begin{equation}\label{eq:harm-2}
\norm{v_j - h}_{W^{1,2}(B_1)} \geq \e_0. 
\end{equation} 
Due to the boundedness of $\{v_j\}_{j=1}^\infty$ in $W^{1,2}(B_1)$, we can assume without loss of generality that $v_j \ra v_0$ weakly in $W^{1,2}(B_1)$ for some $v_0 \in W^{1,2}(B_1)$. Then $v_j\ra v_0$ strongly in $L^2(B_1)$, and $\nabla v_j \ra \nabla v_0$ weakly in $(L^2(B_1))^n$. Hence, passing to the limit in \eqref{eq:vj-apprx-false} and \eqref{eq:vj-apprx-false2} (with $r\in(0,\frac{1}{2})$ fixed) we obtain 
\begin{equation*}
\norm{v_0}_{W^{1,2}(B_1)} \leq 1,\quad\norm{\Delta v_0}_{W^{-1,2}(B_1)} = 0,\quad\text{and}
\end{equation*} 
\begin{equation*}
\sup_{r\in(0,\frac{1}{2})} \frac{ \norm{v_0}_{L^2(B_r)} }{ r^{\frac{n}{2} + d - 1} \theta(r) } \leq 1. 
\end{equation*} 
In particular, $v_0$ is a harmonic function in $B_1$, and in particular, due to the last observation, we must have 
\begin{equation}\label{eq:v0-apprx-false}
\max_{|\beta| \leq d-1} |\partial^\beta v_0 (0)| = 0. 
\end{equation} 

On the other hand, let $h_j$ be a harmonic function on $B_1$ such that $h_j - v_j \in W_0^{1,2}(B_1)$. Then one can deduce from \eqref{eq:vj-apprx-false} and the Poincar\'e inequality that
\begin{equation}\label{eq:hj-apprx-false0}
\norm{h_j - v_j}_{W^{1,2}(B_1)} \leq c_1 \norm{\nabla (h_j - v_j)}_{L^2(B_1)} \leq c_1 \norm{v_j}_{W^{-1,2}(B_1)} \leq \frac{c_1}{j},  
\end{equation} 
for some constant $c_1 > 1$ depending only on $n$. Therefore, 
\begin{equation}\label{eq:hj-apprx-false}
\norm{h_j}_{W^{1,2}(B_1)} \leq \norm{v_j}_{W^{1,2}(B_1)} + \frac{c_1}{j} \leq 2c_1,
\end{equation}  
and the interior gradient estimate for harmonic functions implies that 
\begin{equation}\label{eq:hj-apprx-false2}
\max_{|\beta|\leq d} \sup_{x\in B_{1/2}} | \partial^\beta h_j(x)| \leq c_2, 
\end{equation} 
where $c_2>0$ depends only on $n$ and $d$. Hence, defining $\vp_j$ as the difference between $h_j$ and its $(d-1)$-th order Taylor polynomial at $0$, i.e.,
\begin{equation}\label{eq:vpj}
\vp_j (x) = h_j (x) - \sum_{|\beta|\leq d-1} \frac{\partial^\beta h_j(0)}{|\beta|!} x^\beta, 
\end{equation}
we can verify that $\vp_j$ is also a harmonic function on $B_1$, and \eqref{eq:harm-1} holds with $h = \vp_j$, and $\bar c = 2c_1 + c_3c_2$, where $c_3$ is a constant depending only on $n$ and $d$. In addition, owing to the initial hypothesis, \eqref{eq:harm-2} must hold with $h = \vp_j$ as well.

Note that \eqref{eq:hj-apprx-false} and \eqref{eq:hj-apprx-false2} yield compactness of sequence $\{h_j\}_{j=1}^\infty$, as well as $\{\vp_j\}_{j=1}^\infty$ in $W^{1,2}(B_1)\cap C^d(B_{1/2})$. Hence, after extracting a subsequence if necessary, we can assume without losing any generality that $h_j \ra h_0$ weakly in $W^{1,2}(B_1)$, and $\partial^\beta h_j \ra \partial^\beta h_0$ uniformly on $B_{1/2}$, for any $|\beta|\leq d$, for certain harmonic function $h_0 \in W^{1,2}(B_1)$. 

In particular, recalling that $h_j - v_j \in W_0^{1,2}(B_1)$ and $v_j \ra v_0$ weakly in $W^{1,2}(B_1)$, we deduce $h_0 = v_0$ on $B_1$. This along with \eqref{eq:v0-apprx-false} implies that $|\partial^\beta h_j(0)| \ra 0$ for any $|\beta|\leq d-1$. Therefore, one can deduce from \eqref{eq:hj-apprx-false0} and \eqref{eq:vpj} that for any sufficiently large $j$, 
\begin{equation*}
\begin{split} 
\norm{\vp_j - v_j}_{W^{1,2}(B_1)} &\leq \norm{h_j - v_j}_{W^{1,2}(B_1)} + \norm{\vp_j - h_j}_{W^{1,2}(B_1)} \\
&\leq \frac{c_1}{j} + c_4 \max_{|\beta|\leq d-1} |\partial^\beta h_j(0)| \\
& < \e_0, 
\end{split} 
\end{equation*} 
with a constant $c_4$ depending at most on $n$ and $d$. Hence, we arrive at a contradiction to \eqref{eq:harm-2} with $h = \vp_j$. Thus, there must exist some $\delta$, depending only on $n$, $d$, $\theta$ and $\bar c$; recalling that $\bar c = 2c_1 + c_3c_2$, and tracking down the dependence of $c_1$, $c_2$ and $c_3$, we conclude that the dependence of $\delta$ is restricted to $n$, $d$ and $\theta$, as desired. 
\end{proof} 

Now we are ready to present the point-wise $C^d$-approximation for a weak solution to \eqref{eq:main-sing} around a vanishing point with order at least $d-1$. As in the previous section, we shall prove it in a slightly more general context. 

\begin{proposition}\label{proposition:v-Cd}
Given a Dini modulus of continuity $\omega$, and an integer $d\in\N$, there exist constants $\rho\in(0,\frac{1}{4})$ and $C>0$, depending only on $n$, $\omega$ and $d$, such that if $v\in W^{1,2}(B_1)$ satisfies 
\begin{equation}\label{eq:v-d-0}
\lim_{r\ra 0} \frac{\norm{v}_{L^2(B_r)}}{r^{\frac{n}{2} + d - 1}} = 0, \quad\text{and}
\end{equation} 
\begin{equation}\label{eq:v-pde-d}
\max\left\{ \norm{v}_{W^{1,2}(B_1)} , \sup_{r\in(0,1)} \frac{\norm{\Delta v}_{W^{-1,2}(B_r)}  }{\omega(\rho r) ( \norm{\nabla v}_{L^2(B_r)} + r^{\frac{n}{2} + d - 1} )}\right\} \leq 1, 
\end{equation} 
then there is a homogeneous harmonic polynomial $P$ of degree $d$ such that 
\begin{equation}\label{eq:v-Cd}
\sup_{|\beta| = d}|\partial^\beta P| + \sup_{r\in(0,1)} \frac{\norm{v - P}_{L^2(B_r)} + r\norm{\nabla (v - P)}_{L^2(B_r)}}{r^{\frac{n}{2} + d}\omega_1(r)} \leq C_d,
\end{equation} 
where $C_d$ is a positive constant depending only on $n$, $\lambda$, $\omega$ and $d$, and $\omega_1$ is a modulus of continuity determined by $\omega$ only. 
\end{proposition}

\begin{remark}\label{remark:v-Cd}
Modulus of continuity $\omega_1$ is given as the same as in Remark \ref{remark:v-C1}.
\end{remark} 

\begin{proof}[Proof of Proposition \ref{proposition:v-Cd}]
We are going to prove this lemma by an induction argument with respect to order $d$. Note that the initial case $d=1$ is already established in Proposition \ref{proposition:v-C1}. Henceforth, given any $d\in\N$ with $d\geq 2$, we shall assume that this lemma is true for order $d-1$, for certain constants $\rho' \in (0,\frac{1}{2})$ and $C'>0$, depending only on $n$, $\omega$ and $d-1$, and we will prove that it continues to hold with order $d$. 

Let $\rho\in(0,\rho']$, $\mu\in (0,\frac{1}{2})$ and $\bar{c} \geq 1$ be some constants to be determined, by $n$, $\omega$, and $d$ only. We are going to construct a sequence $\{P_k\}_{k=0}^\infty$ of homogeneous harmonic polynomials of degree $d$ such that for each $k\in\N\cup\{0\}$, we have 
\begin{equation}\label{eq:v-Pk-W12}
\norm{v -P_k}_{L^2(B_{\mu^k})} + \mu^k \norm{\nabla (v - P_k)}_{L^2(B_{\mu^k})} \leq \frac{\mu^{k(\frac{n}{2} + d)}\omega(\rho\mu^m)}{\omega(\rho)},\quad\text{and}
\end{equation}
\begin{equation}\label{eq:Pk}
\sup_{|\beta| = d} |\partial^\beta P_k - \partial^\beta P_{k-1}| \leq \frac{c_0\mu^{k(\frac{n}{2} + d)} \omega(\rho\mu^m)}{\omega(\rho)},
\end{equation}  
with $P_{-1} = 0$, where $c_0>0$ depends only on $n$ and $d$. 

Set $p_0 = 0$, so that \eqref{eq:Pk} is trivial for $k=0$, and \eqref{eq:v-Pk-W12} holds for $k=0$ from the assumption $\norm{v}_{W^{1,2}(B_1)}\leq 1$. Now let us assume that these hypotheses are verified, with for all $k\leq m$, for some $m\in\N\cup\{0\}$ given, and let us attempt to prove them for $k= m+1$. 

Define $v_m \in W^{1,2}(B_1)$ by 
\begin{equation}\label{eq:vm-d}
v_m(x) = \frac{\omega(\rho)}{\mu^{dm} \omega(\rho\mu^m)} (v - P_m)(\mu^m x).
\end{equation}
Due to induction hypothesis \eqref{eq:v-Pk-W12} for $k=m$, we have $\norm{v_m}_{W^{1,2}(B_1)} \leq 1$. Moreover, since $P_m$ is harmonic, following the computation in \eqref{eq:vm-pde}, we can also deduce from \eqref{eq:v-pde-d} that
\begin{equation*}
\begin{split}
\norm{\Delta v_m}_{W^{-1,2}(B_r)} &= \frac{\omega(\rho)}{\mu^{(\frac{n}{2} + d-1)m} \omega(\rho\mu^m)} \norm{\Delta v}_{W^{-1,2}(B_{\mu^m r})} \\
&\leq \frac{\omega(\rho)\omega(\rho\mu^m r)}{\omega(\rho\mu^m)} \left( \frac{\norm{\nabla v}_{L^2(B_{\mu^m r})}}{\mu^{(\frac{n}{2} + d - 1)m}} + r^{\frac{n}{2} + d - 1}\right) \\
&\leq \omega(\rho \mu^m r) \norm{\nabla v_m}_{L^2(B_r)} + \omega(\rho)\left( \norm{\nabla P_m}_{L^2(B_r)} + r^{\frac{n}{2} + d - 1} \right).
\end{split}
\end{equation*}
Since \eqref{eq:Pk} holds for all $0\leq k\leq m$, we have $|D^d P_m|\leq c_0 \omega(\rho)^{-1}\psi(\rho)$, with $\psi(\rho) = \sum_{k=0}^\infty \omega(\rho \mu^k)$, and since $\omega(\rho) < \psi(\rho)$, we have
\begin{equation}\label{eq:vm-pde-d}
\begin{split}
\norm{\Delta v_m}_{W^{-1,2}(B_r)} &\leq \omega(\rho r)\norm{\nabla v_m}_{L^2(B_r)} + \left(c_0\sqrt{\beta_n} + 1\right) \psi(\rho) r^{\frac{n}{2} + d - 1} \\
& \leq c_1\omega(\rho r) \left( \norm{\nabla v_m}_{L^2(B_r)} + r^{\frac{n}{2} + d - 2}\right), 
\end{split}
\end{equation} 
for any $r\in(0,1)$, for certain $c_1>1$ satisfying $(c_0\sqrt{\beta_n} + 1) \psi(\rho) r \leq c_1 \omega(\rho r)$ for any $r\in(0,1)$; note that $c_1$ depends only on $n$, $c_0$, $\omega$ and $\rho$, hence on $n$, $\omega$ and $d$ only, as we are going to choose $\rho$ to be determined by $n$, $\omega$ and $d$ only.

Recall that we set our initial induction hypothesis as that this proposition holds for order $d-1$, with constants $\rho'$ and $C'$ depending only on $n$, $\omega$ and $d-1$. Hence, let us choose $\rho\in(0,\rho']$ sufficiently small such that $c_1\omega(\rho r) \leq \omega(\rho' r)$, with $c_1$ as above. 

Now in view of the second inequality in \eqref{eq:vm-pde-d} and $\norm{v_m}_{W^{1,2}(B_1)} \leq 1$, we deduce from the initial induction hypothesis that there is a homogeneous harmonic polynomial $P'$ of degree $d-1$ for which 
\begin{equation}\label{eq:vm-P'-W12}
\norm{v_m - P'}_{L^2(B_r)} + r \norm{\nabla (v_m - P')}_{L^2(B_r)} \leq C' r^{\frac{n}{2} + d - 1} \omega_1(r), 
\end{equation} 
for any $r\in(0,1)$. However, due to \eqref{eq:v-d-0}, we have $\norm{v_m}_{L^2(B_r)} = o(r^{\frac{n}{2} + d- 1})$ as $r\ra 0$ as well, which along with \eqref{eq:vm-P'-W12} implies that $P' = 0$, as $P'$ being a homogeneous polynomial of degree $d-1$. Thus, we have
\begin{equation}\label{eq:vm-L2-d1}
\norm{v_m}_{L^2(B_r)} + r \norm{\nabla v_m}_{L^2(B_r)} \leq C' r^{\frac{n}{2} + d - 1} \omega_1(r),
\end{equation} 
for any $r\in(0,1)$. 

With \eqref{eq:vm-L2-d1} and \eqref{eq:vm-pde-d} at hand, we can employ Lemma \ref{lemma:apprx-hom} with $v = v_m$, and $\theta = C'\omega_1$. Hence, given $\e\in(0,1)$ to be determined later by $n$, $\omega$ and $d$, we can find a harmonic function $h \in W^{1,2}(B_1)$ such that 
\begin{equation*}
\max_{|\beta|\leq d-1} |\partial^\beta h(0)| = 0,\quad \norm{h}_{W^{1,2}(B_1)} \leq C_1\quad\text{and}
\end{equation*}
\begin{equation}\label{eq:vm-h-W12}
\norm{ v_m - h }_{W^{1,2}(B_1)} \leq \e,
\end{equation} 
with $C_1$ depending only on $n$ and $d$, provided that we choose $\rho$ sufficiently small so as to satisfy $c_1 \omega(\rho) \leq \delta$, with $\delta$ being the small constant chosen as in Lemma \ref{lemma:apprx-hom}, depending only on $n$, $\omega$, $d$, $\theta$ and $\e$, so on $n$, $\omega$ and $d$ only. 

The rest of the proof follows the standard argument. Since $h$ is a harmonic function on $B_1$, we have $\norm{h}_{W^{d+1,\infty}(B_{1/2})} \leq C_2 \norm{h}_{L^2(B_1)} \leq C_2C_1$, where $C_2>0$ depends only on $n$ and $d$. As $h$ vanishing with order at least $d-1$ at the origin, the $d$-th order Taylor polynomial, $Q_m$, of $h$, should satisfy $|\partial^\beta Q_m (0)| = 0$ for any $|\beta|\leq d-1$, and $\partial^\beta Q_m = \partial^\beta h(0)$ for any $|\beta| = d$. Then $Q_m$ is a homogeneous harmonic polynomial of degree $d$, and with $\mu\in (0,\frac{1}{2})$, 
\begin{equation}\label{eq:h-Qm-W12-d}
\norm{Q_m}_{L^\infty(B_1)} + \frac{\norm{h - Q_m}_{L^\infty(B_\mu)} + \mu \norm{\nabla ( h - Q_m)}_{L^\infty(B_\mu)} }{\mu^{\frac{n}{2} + d+1}} \leq C_3, 
\end{equation} 
where $C_3>0$ depends only on $n$ and $d$. Hence, it follows from \eqref{eq:vm-h-W12} and \eqref{eq:h-Qm-W12-d} that
\begin{equation}\label{eq:vm-Qm-L2-d}
\norm{ v_m - Q_m }_{L^2(B_\mu)} \leq \e + C_3\sqrt{\beta_n} \mu^{\frac{n}{2} + d + 1} \leq \frac{\mu^{\frac{n}{2} + d}\omega(\mu)}{2\omega(1)},
\end{equation}
and similarly,
\begin{equation}\label{eq:Dvm-DQm-L2-d}
\norm{ \nabla (v_m - Q_m) }_{L^2(B_\mu)} \leq  \frac{\mu^{\frac{n}{2} + d - 1}\omega(\mu)}{2\omega(1)},
\end{equation} 
by first selecting $\mu$ and $\e$ sufficiently small such that
\begin{equation*}
4C_3\sqrt{\beta_n} \omega(1) \mu \leq \omega(\mu),
\end{equation*} 
and selecting $\e$ small accordingly so as to satisfy
\begin{equation*}
4\omega(1) \e \leq \mu^{\frac{n}{2} + d} \omega(\mu). 
\end{equation*}
Clearly, $\mu$ and $\e$ are determined by $n$, $\omega$ and $d$ only, so is $\rho$ as chosen above. 

Finally, we set
\begin{equation*}
P_{m+1} = P_m + \frac{\mu^{dm} \omega(\rho \mu^m)}{\omega(\rho)} Q_m\left( \frac{\cdot}{\mu^m}\right),
\end{equation*}
so that \eqref{eq:v-Pk-W12} and \eqref{eq:Pk} continue to hold for $k = m+1$, due to \eqref{eq:vm-Qm-L2-d}, \eqref{eq:Dvm-DQm-L2-d} and \eqref{eq:h-Qm-W12-d} with $c_0 = (d!)^{-1}C_3$. This finishes the proof. 
\end{proof} 

As in the previous section, we deduce a uniform $C^d$-approximation of $u^+ - \kappa(z) u^-$ at any interior point $z$, for a weak solution $u$ to \eqref{eq:main-sing}, as a corollary to Proposition \ref{proposition:v-Cd}. Again the key is the uniformly continuous dependence of the approximating homogeneous harmonic polynomials on the singular points. 

\begin{lemma}\label{lemma:int-Cd}
Under the setting of Theorem \ref{theorem:sing}, for each $d\in\N$ with $d\geq 2$, and every $z\in\Gamma_{d-1}(u)$, there exists a (unique) homogeneous $a_+(z)$-harmonic polynomial $P_z$, of degree $d$, such that
\begin{equation}\label{eq:phiz-Pz-W12}
\begin{split}
& \sup_{|\beta| = d} |\partial^\beta P_z| + \sup_{r\in (0,\frac{1}{4})} \frac{\norm{v_z - P_z}_{L^\infty(B_r(z))}}{r^d\omega_1(r)} \\
&\quad + \sup_{r\in (0,\frac{1}{4})} \frac{\norm{\nabla (v_z - P_z)}_{L^p(B_r(z))}}{r^{\frac{n}{p} + d} \omega_1(r)} \leq C_{d,p}\norm{u}_{L^2(B_1)} ,
\end{split}
\end{equation} 
where $v_z = u^+ - \kappa(z) u^-$, $C_p$ is a positive constant depending only on $n$, $\lambda$, $\omega$, $d$ and $p$, and $\omega_1$ is a modulus of continuity determined by $\omega$ only. Moreover, for any pair $(z,z')$ of distinct points in $\Gamma_{d-1}(u)\cap \overline{B_{1/2}}$, 
\begin{equation}\label{eq:DdPz-C}
\sup_{|\beta| = d}\frac{|\partial^\beta P_z - \partial^\beta P_{z'}|}{\omega_1(|z-z'|)} \leq C_d\norm{u}_{L^2(B_1)},
\end{equation}  
where $C_d$ is a positive constant depending only on $n$, $\lambda$, $\omega$ and $d$. 
\end{lemma}

\begin{proof}
The proof is the same as that of Lemma \ref{lemma:int-C1} and Lemma \ref{lemma:int-C1-more}, so we omit the details. 
\end{proof} 

Due to this lemma, we have that $\Gamma_{d-1}(u)$ is closed, and contains $\cS_d(u)$ as a relatively open set. For the future reference, let us record it as a lemma, and also give a proof for the sake of completeness. 

\begin{lemma}\label{lemma:closed}
Under the setting of Theorem \ref{theorem:sing}, $\Gamma_{d-1}(u)$ is a closed set in $B_1$, and $\cS_d(u)$ is a relatively open set in $\Gamma_{d-1}(u)$. 
\end{lemma} 

\begin{proof}
Recall from Theorem \ref{theorem:fb-reg} that $\{u=0\}$ is closed in $\Omega$, and $\cN(u)$, the regular part of $\{u=0\}$ defined as in \eqref{eq:reg}, is relatively open in $\{u=0\}$. Hence, $\cS(u) = \{u=0\}\setminus \cN(u)$ is a closed in $B_1$. Since $\Gamma_1(u) = \cS(u)$, we have proved that $\Gamma_{d-1}(u)$ is a closed set in $B_1$ for $d = 2$.

Now let $d\in\N$ be arbitrary number with $d\geq 2$, and assume that $\Gamma_{d-1}(u)$ is a closed set in $\Omega$. Given $z\in\Gamma_{d-1}(u)$, let $P_z$ be as in Lemma \ref{lemma:int-Cd}; note that after a standard scaling argument, one can obtain $P_z$ for any $z\in \Gamma_{d-1}(u)\cap B_1$. Suppose that $z\in\cS_d(u)$. Then according to \eqref{eq:phiz-Pz-W12}, $|D^d P_z| > 0$, and there exists a small neighbourhood $B_r(z)$ such that $|D^d P(z')| > 0$ for any $z' \in \Gamma_{d-1} (u) \cap B_r(z)$. Hence, again by \eqref{eq:phiz-Pz-W12}, $z' \in \cS_d(u)$ for all $z'\in \Gamma_{d-1}(u)\cap B_r(z)$, proving that $\cS_d(u)$ is relatively open in $\Gamma_{d-1}(u)$. 

Consequently, as $\Gamma_d(u) = \Gamma_{d-1}(u)\setminus \cS_d(u)$, $\Gamma_d(u)$ also becomes a closed set in $\Omega$. By the induction principle, the conclusion of the lemma follows. 
\end{proof}

Now we are in a position to analyse the structure of the  set, $\cS_d(u)$, of vanishing points of $u$ with order $d$. Note that $\cS_d(u)$ is a subset of $\Gamma_{d-1}(u)$ that consists precisely of the points $z$ satisfying $|D^d P_z| > 0$. 

As $P_z$ being a (nontrivial) homogeneous $a_+(z)$-harmonic polynomial of degree $d$, $\cS_d(P_z)$ is a linear subspace of $\R^n$ with dimension at most $n-2$, and $P_z$ is invariant under translation in $\cS_d(P_z)$ (c.f. see page 992 in \cite{Han}); it is the homogeneity of $P_z$ that makes $\cS_d(P_z)$ a linear subspace, and then the fact that $P_z$ is $a_+(z)$-harmonic yields $\dim(\cS_d(P_z))\leq n-2$. Following classical literature (c.f. \cite{Han}), we shall also decompose $\cS_d(P_z)$ into  
\begin{equation}\label{eq:sing-d-j}
\cS_d^j (u) = \{ z\in\cS_d(u): \dim(\cS_d(P_z)) = j \},\quad\text{for } j=0,1,\cdots,n-2. 
\end{equation} 

\begin{lemma}\label{lemma:delta}
Let $P$ be a nontrivial homogeneous polynomial of degree $d\geq 2$ on $\R^n$, with $\dim(\cS_d(P)) = j$. Then there exists a set  $\{\beta_l : l = 1,2,\cdots,n-j\}\subset (\N\cup\{0\})^d$ of multi-indices with $|\beta_l| = d-1$ such that matrix $(\partial_k\partial^{\beta_l} P)$ of dimension $n(n-j)$ has full rank. 
\end{lemma}

\begin{proof}
Since $\dim(\cS_d(P)) = j$ and $P$ is invariant under translation in $\cS_d(P)$, we can assume without loss of generality that $\cS_d(P) = \{0\}^{n-j}\times \R^{n-j}$, and $P$ is a function of $(x_1,\cdots,x_{n-j})$ only. Let $I$ be the set of all multi-indices $\beta$ with $|\beta| = d-1$ and $\beta\cdot e_i = 0$ for all $n-j+1 \leq i\leq n$, and write $N$ by the cardinality of $I$. Rewriting $I$ by $\{\beta_1,\cdots,\beta_N\}$, we can consider a $((n-j)\times N)$-dimensional matrix $A$ whose $(k,l)$-th entry is given by $\partial_k \partial^{\beta_l} P$, for any $k\in\{1,\cdots,n - j\}$ and $l \in \{1,\cdots,N\}$.  Note that $\partial_k \partial^{\beta_l} P$ is a constant for any $(k,l)$, as $P$ being a polynomial of degree $d$. 

This lemma will follow straightforwardly if matrix $A$ has full rank, as $N\geq n-j$. For this reason, let us assume towards a contradiction that the rank of $A$ is at most $n-j-1$. Then there must exist a nonzero vector $(\nu_1,\cdots,\nu_{n-j}) \in \R^{n-j}$ such that
\begin{equation*}
\nu_1\partial_1 \partial^{\beta_l} P + \cdots + \nu_{n-j} \partial_{n-j} \partial^{\beta_l} P = 0,
\end{equation*}
for any $l=1,\cdots,N$. However, since $\partial^\beta P$ is a linear function for any multi-index $\beta\in I$ (as $P$ being a polynomial of degree $d$), the above identity implies that with $\nu = (\nu_1,\cdots,\nu_{n-j},0,\cdots,0)\in\R^n$,
\begin{equation*}
\partial^{\beta_l} P(t\nu) = t \partial^{\beta_l} P(\nu) = 0,
\end{equation*} 
for any $l=1,\cdots,N$ and any $t\in\R$. From the definition of $I = \{\beta_1,\cdots,\beta_N\}$, we have proved that $|D^{d-1} P| = 0$ on line $L(\nu) = \{t\nu : t\in\R\}$. 

However, if $|D^i P| = 0$ on $L(\nu)$, for some $i\in\{1,\cdots,d-1\}$, then for any multi-index $\gamma$ with $|\gamma| = i-1$ and $\gamma \cdot e_i = 0$ for $n-j +1 \leq i\leq n$, we have 
\begin{equation}
\partial^\gamma P(t\nu) = t\int_0^1 \nu_i \partial_i \partial^\gamma P(ts\nu)\,ds = 0,
\end{equation} 
where we used $\partial^\gamma P (0) = 0$, which follows from the homogeneity of $P$. That is, $|D^{i-1} P | = 0$ on $L(\nu)$ as well.

Iterating this observation, we obtain that $|D^i P| = 0$ on $L(\nu)$ for any $i=0,1,\cdots,d-1$, proving that $L(\nu)\subset\cS_d(P)$, a contradiction to the earlier assumption that $\cS_d(P) = \{0\}^{n-j}\times \R^j$. Hence, the lemma is proved. 
\end{proof} 

To this end, we are ready to prove the second part Theorem \ref{theorem:sing}, concerning the structure of the singular set of weak solutions to \eqref{eq:main-sing}. Let us remark that with the aid of Lemma \ref{lemma:delta}, the proof only involves two fundamental tools from multivariable calculus and real analysis, which are Whitney's extension theorem and the implicit function theorem. Nevertheless, we obtain a more refined structure for the lower dimensional singular set, i.e., $\cS_d^j(u)$ with $0\leq j\leq n-3$, compared to the classical results, such as \cite{Han}. 

\begin{proof}[Proof of Theorem \ref{theorem:sing}]
The first part \eqref{eq:sing-decom} of this theorem concerning the integral vanishing order of singular points follows straightforwardly from Lemma \ref{lemma:int-Cd}, since the lemma yields $\Gamma_{d-1}(u) = \cS_d(u)\cap \Gamma_d(u)$, for any $d\in\N$ with $d\geq 2$, where $\Gamma_d(u)$ and $\cS_d(u)$ are defined as in \eqref{eq:zero-d} and \eqref{eq:sing-d}. . 

In order to prove the second part \eqref{eq:sing-decom-d}, let us first fix $d\in\N$ with $d\geq 2$. Let $K$ be a compact set contained in $B_1$. By Lemma \ref{lemma:closed}, $\Gamma_{d-1}(u)\cap K$ is also compact. Due to Lemma \ref{lemma:int-Cd}, one can construct a function $g\in C^{d,\omega_1}(\R^n)$, via Whitney's extension theorem \cite{whit,gla}, as in the proof of Theorem \ref{theorem:fb-reg}, such that 
\begin{equation}\label{eq:g}
\Gamma_{d-1}(u)\cap K \subset \bigcap_{i=0}^{d-1} \bigcap_{|\beta| = i}\{ |\partial^\beta g| = 0\},\quad\text{and}
\end{equation}
\begin{equation}\label{eq:Ddg}
\max_{|\beta| = d} |\partial^\beta g(z) - \partial^\beta P_z| = 0 \quad\text{for any }z\in\Gamma_{d-1}(u)\cap K,
\end{equation}
where $P_z$ is the (unique) homogeneous $a_+(z)$-harmonic polynomial of degree $d$ approximating $v_z = u^+ - \kappa(z) u^-$ in the sense of \eqref{eq:phiz-Pz-W12}. Let us also address a more recent work \cite{Fef09} by C. Fefferman concerning a general extension from $C^{d,\theta}(E)$ to $C^{d,\theta}(\R^n)$, for any compact set $E$ and any modulus of continuity $\theta$. 

Now let $z\in\cS_d^j(u)\cap K$, for some $j\in\{0,1,\cdots,n-2\}$, where $\cS_d^j(u)$ is given as \eqref{eq:sing-d-j}. Then $P_z$ is a non-trivial homogeneous polynomial with $\dim(\cS_d(P_z)) = j$. Hence, Lemma \ref{lemma:delta} along with \eqref{eq:Ddg} gives us a set $\{\beta_l :l=1,2,\cdots,n-j\}\subset (\N\cup\{0\})^d$ of multi-indices with $|\beta_l| = d-1$ such that the Jacobian of the vector-valued mapping $(\partial^{\beta_1} g, \cdots, \partial^{\beta_{n-j}} g): \R^n \ra \R^{n-j}$ has full rank at $z$. On the other hand, from \eqref{eq:g}, we have 
\begin{equation}\label{eq:g-re}
\Gamma_{d-1}(u)\cap K\subset \bigcap_{l=1}^{n-j} \{ \partial^{\beta_l} g = 0\}.
\end{equation} 
Since $\partial^{\beta_l} g \in C^{1,\omega_1}(\R^n)$ for any $l=1,\cdots,n-2$, it follows from the implicit function theorem that $\bigcap_{l=1}^{n-j} \{ \partial^{\beta_l} g = 0\}$ can be represented by a  $j$-dimensional $C^{1,\omega_1}$-graph in a small neighbourhood of $z$. 

Since the above argument works for any $z\in\cS_d^j(u)\cap K$, and since $K$ is compact and $\cS_d^j(u)$ is relatively open in $\Gamma_{d-1}(u)$, one can use a standard covering argument to conclude that $\cS_d^j(u)\cap K$ is contained on a countable union of $j$-dimensional $C^{1,\omega_1}$-manifolds. Finally, the proof is finished by setting $\cS^j(u) = \bigcup_{d=2}^\infty\cS_d^j(u)$. 
\end{proof}


\section{Doubling Property}\label{section:double}

In this section, we shall study the doubling property of a weak solution to \eqref{eq:main-sing}. The main purpose of this section is to obtain a uniform doubling property depending on the maximal vanishing order among all singular points, provided that the solution does not vanish at any point with infinite order. More precisely, we assert the following.

\begin{theorem}\label{theorem:double}
Assume that $a_+,a_- \in (C^{0,\omega}(B_1))^{n^2}$ satisfy \eqref{eq:a-ellip}, \eqref{eq:a-Dini} and \eqref{eq:a-k}. Let $u \in W^{1,2}(B_1)$ be a nontrivial weak solution to \eqref{eq:main-sing} and suppose, for some $d\in\N$, that $u$ vanishes on $\{u=0\}$ with order $d$ at most. Then there exists $r_{d,u} \in (0,\frac{1}{4})$, depending only on $n$, $\lambda$, $\omega$, $d$ and certain character of $u$, such that
\begin{equation}\label{eq:double}
\sup_{z\in \cB_{d,u}}\sup_{r\in(0,r_{d,u}]} \frac{\norm{u}_{L^2(B_{2r}(z))}}{\norm{u}_{L^2(B_r(z))}} \leq 2^{c_1d + c_2 + \frac{n}{2}}, 
\end{equation}
where $c_1>1$ is an absolute constant, $c_2>0$ is a constant depending only on $n$ and $\lambda$, and $\cB_{d,u} = \bigcup_{x\in \{u=0\}\cap \overline{B_{1/2}}} B_{r_{d,u}}(x)$. 
\end{theorem} 

The following lemma is concerned with a doubling condition for harmonic functions, which is proved in Theorem A.3 in \cite{HS}. Here we shall present the statement and refer to \cite{HS} for the proof. 

\begin{lemma}[Theorem A.3 in {\cite{HS}}]\label{lemma:double-harm}
Let $h$ be a harmonic function in $B_R$. Then for any $\alpha$, $\beta$ with $\alpha>\beta>1$, and any $q\geq \frac{1}{2}$, with $q\not\in\N$, there exists $\e>0$, depending only on $\alpha$, $\beta$, and $\dist(q,\N)$, such that for any $r,s,t \in(0,R)$ with $r<s<t$ and $\frac{s}{r} , \frac{t}{s} \in [\alpha,\beta]$, if  
\begin{equation*}
\norm{h}_{L^2(B_t)} \leq \left( \frac{t}{s} \right)^{\frac{q + n}{2}} \norm{h}_{L^2(B_s)},
\end{equation*}
then 
\begin{equation*}
\norm{h}_{L^2(B_s)}  \leq \left( \frac{s}{r} \right)^{\frac{q - \e + n}{2}} \norm{h}_{L^2(B_r)}.
\end{equation*} 
\end{lemma} 

Now we assert that if a weak solution to a broken quasilinear PDE \eqref{eq:main-sing} enjoys doubling property at certain scale, then the property is satisfied for any small scale, provided that the coefficients from each side are almost constant.

\begin{lemma}\label{lemma:double}
Given any integer $d\in\N$ and any real $\gamma\in(0,\frac{1}{4})$, there exists a constant $\delta>0$, depending only on $n$, $\lambda$, $d$ and $\gamma$, such that if $a_+,a_-\in (L^\infty(B_1))^{n^2}$ satisfy \eqref{eq:a-ellip} and $\max\{\norm{a_+ - I}_{L^\infty(B_1)}, \norm{a_- - \kappa I}_{L^\infty(B_1)}\}\leq \delta$ for some real constant $\kappa\in[\lambda^2,\lambda^{-2}]$, and if  $u\in W^{1,2}(B_1)$ is a weak solution to \eqref{eq:main-sing} such that 
\begin{equation}\label{eq:double-initial}
\norm{v}_{L^2(B_1)} \leq 2^{d + \gamma + \frac{n}{2}} \norm{v}_{L^2(B_{1/2})},
\end{equation} 
with $v = u^+ - \kappa u^-$, then  
\begin{equation}\label{eq:double-inside}
 \norm{v}_{L^2(B_{1/2})} \leq 2^{d + \gamma+ \frac{n}{2}}  \norm{v}_{L^2(B_{1/4})}. 
\end{equation} 
\end{lemma}

\begin{proof}
Let $\tau\in(0,\frac{1}{8})$ be a sufficiently small real to be determined later, and let $h$ be the harmonic function on $B_{1-\tau}$ such that $h- v \in W_0^{1,2}(B_{1-\tau})$. In view of \eqref{eq:main-sing}, we realise that $h-v\in W_0^{1,2}(B_{1-\tau})$ is a weak solution to 
\begin{equation*}
\Delta (h - v) = \ddiv ( (a_+ - I)\nabla u^+ - (a_- - \kappa I)\nabla u^-).
\end{equation*} 
Hence, it follows from the closeness condition on $a_+$ and $a_-$ that 
\begin{equation}\label{eq:Dh-Dv-L2}
\norm{\nabla (h-v)}_{L^2(B_{1-\tau})} \leq \delta \norm{\nabla u}_{L^2(B_{1-\tau})} \leq \frac{c_0\delta}{\tau} \norm{u}_{L^2(B_1)},
\end{equation}
where the last inequality is deduced from the local energy estimate, as $u$ being a weak solution to a uniformly elliptic PDE; here $c_0$ is a constant depending only on $n$ and $\lambda$. 

Since $\kappa$ is a real constant satisfying $\lambda^2\leq \kappa\leq \lambda^{-2}$, we have $\norm{u}_{L^2(B_1)}\leq \lambda^{-2} \norm{v}_{L^2(B_1)}$. Inserting this inequality into \eqref{eq:Dh-Dv-L2}, and utilising the doubling condition \eqref{eq:double-initial}, we obtain 
\begin{equation}\label{eq:Dh-Dv-L2-re}
\norm{\nabla (h-v)}_{L^2(B_{1-\tau})} \leq \frac{c_1\delta}{\tau} 2^{d + \gamma + \frac{n}{2}} \norm{v}_{L^2(B_{1/2})},
\end{equation} 
where $c_1 = \lambda^{-2} c_0$. Recalling that $h-v\in W_0^{1,2}(B_{1-\tau})$, we can deduce from the Poincar\'e inequality that 
\begin{equation}\label{eq:h-v-L2}
\norm{h-v}_{L^2(B_{1-\tau})} \leq \frac{c_2c_1\delta}{\tau} 2^{d + \gamma + \frac{n}{2}} \norm{v}_{L^2(B_{1/2})},
\end{equation} 
where $c_2$ is a constant depending only on $n$. 

We can deduce from \eqref{eq:h-v-L2} and the doubling condition \eqref{eq:double-initial} that
\begin{equation*}
\begin{split}
\norm{h}_{L^2(B_{1-2\tau})} &\leq \norm{v}_{L^2(B_1)} +  \frac{c_2c_1\delta}{\tau} 2^{d + \gamma + \frac{n}{2}} \norm{v}_{L^2(B_{1/2})} \\
&\leq \left( 1+ \frac{c_2c_1\delta}{\tau}  \right) 2^{d + \gamma + \frac{n}{2}} \norm{v}_{L^2(B_{1/2})}.
\end{split}
\end{equation*} 
However, since we also have from \eqref{eq:h-v-L2} that 
\begin{equation*}
\begin{split}
\norm{h}_{L^2(B_{1/2})} &\geq \norm{v}_{L^2(B_{1/2})} - \norm{v- h}_{L^2(B_{1/2})} \\
&\geq \left(1 - \frac{c_2c_1\delta}{\tau} 2^{d + \gamma + \frac{n}{2}} \right) \norm{v}_{L^2(B_{1/2})},
\end{split} 
\end{equation*} 
combining the last two estimates, we arrive at
\begin{equation}\label{eq:h-doubling-1}
\norm{h}_{L^2(B_{1-2\tau})} \leq \left( \frac{ 1+ \frac{c_2c_1\delta}{\tau} }{1 - \frac{c_2c_1\delta}{\tau} 2^{d + \gamma + \frac{n}{2}} } \right)2^{d + \gamma + \frac{n}{2}}\norm{h}_{L^2(B_{1/2})}.
\end{equation} 

Given any $\eta\in(0,\frac{1}{8})$, one can find a sufficiently small $\delta>0$, depending only on $c_1$, $c_2$, $d$, $n$, $\gamma$, $\tau$ and $\eta$, hence on $n$, $\lambda$, $d$, $\gamma$ and $\eta$, such that
\begin{equation}\label{eq:delta}
\frac{ 1+ \frac{c_2c_1\delta}{\tau} }{1 - \frac{c_2c_1\delta}{\tau} 2^{d + \gamma + \frac{n}{2}} } \leq 2^{\frac{\eta}{4}}. 
\end{equation} 
On the other hand, we can also choose a small $\tau\in(0,\frac{1}{8})$, depending only on $d$, $\gamma$, $n$ and $\eta$, such that
\begin{equation}\label{eq:tau}
2^{d + \gamma + \frac{\eta}{4} + \frac{n}{2}} \leq (2-4\tau)^{d+ \gamma + \eta + \frac{n}{2}}.
\end{equation} 

To select $\eta$, let us find the constant $\e$ from Lemma \ref{lemma:double-harm} corresponding to $\alpha = \frac{3}{2}$, $\beta = 2$ and any $q$ satisfying $\dist(q,\N) \geq 2\gamma$. Clearly, such a constant $\e$ depends only on $\gamma$ and, in particular, it is independent of the choice of $q$, as long as it satisfies $\dist(q,\N)\geq 2\gamma$. Thus, we can choose $\eta\in(0,\min\{1-4\gamma,\frac{\e}{3}\})$, depending only on $\gamma$, such that with $q = 2(d + \gamma + \eta)$ (so that $\dist(q,\N)\geq 2\gamma$), we have 
\begin{equation}\label{eq:eta}
q - \e = 2(d+ \gamma + \eta) + 2\eta - \e < 2(d + \gamma) - \eta.
\end{equation} 

This shows that $\tau$ determined as in \eqref{eq:tau} now depends on $n$, $d$ and $\gamma$ only, so $\delta$ as in \eqref{eq:delta} depends on $n$, $\lambda$, $d$ and $\gamma$ only. Collecting \eqref{eq:delta} and \eqref{eq:tau}, we deduce from \eqref{eq:h-doubling-1} that
\begin{equation}\label{eq:h-doubling-2} 
\norm{h}_{L^2(B_{1-2\tau})} \leq (2-4\tau)^{d+ \gamma + \eta + \frac{n}{2}} \norm{h}_{L^2(B_{1/2})}.
\end{equation} 

Recalling the choice of $\eta$ (and $\e$) above, we can apply Lemma \ref{lemma:double} to $h$, with $R = 1 - \tau$, $t = 1-2\tau$, $s = \frac{1}{2}$, $r=\frac{1}{4}$, $\alpha = \frac{3}{2}$, $\beta = 2$ and $q = 2(d + \gamma + \eta)$, and obtain from \eqref{eq:eta} that
\begin{equation}\label{eq:h-doubling-3}
\norm{h}_{L^2(B_{1/2})} \leq 2^{d + \gamma - \eta + \frac{n}{2}} \norm{h}_{L^2(B_{1/4})}.
\end{equation} 

Finally, we can repeat as in the derivation of \eqref{eq:h-doubling-1}, and compute that
\begin{equation*}
\begin{split}
\norm{v}_{L^2(B_{1/2})} &\leq \norm{h}_{L^2(B_{1/2})} + \norm{v-h}_{L^2(B_{1/2})} \\
& \leq 2^{d + \gamma - \eta + \frac{n}{2}} \norm{h}_{L^2(B_{1/4})} + \frac{c_2c_1\delta}{\tau} 2^{d+ \eta + \frac{n}{2}} \norm{v}_{L^2(B_{1/2})} \\
& \leq 2^{d + \gamma - \eta + \frac{n}{2}} \norm{v}_{L^2(B_{1/4})} + \frac{c_2c_1\delta}{\tau} 2^{2d+ 3\eta + n} \norm{v}_{L^2(B_{1/2})},
\end{split} 
\end{equation*}
hence, arriving at 
\begin{equation*}
\norm{v}_{L^2(B_{1/2})} \leq \frac{2^{d + \gamma - \eta + \frac{n}{2}} }{1 -\frac{c_2c_1\delta}{\tau} 2^{2d+ 3\eta  + n} }\norm{v}_{L^2(B_{1/4})}. 
\end{equation*}  
With $\tau$ being fixed as in \eqref{eq:tau}, we can adjust $\delta$ to be slightly smaller, if necessary and without affecting its dependence on $n$, $\lambda$, $d$ and $\gamma$ (as $\eta$ depending on $\gamma$ only), such that 
\begin{equation*}
\norm{v}_{L^2(B_{1/2})} \leq 2^{d + \gamma + \frac{n}{2}} \norm{v}_{L^2(B_{1/4})}.
\end{equation*}
This finishes the proof. 
\end{proof} 

Now if our coefficient matrices, $a_+$ and $a_-$, are Dini continuous, and $v = u^+ - \kappa(0) u^-$ satisfies a doubling property at a sufficiently small scale, then one can iterate Lemma \ref{lemma:double} to justify that $v$ satisfies the same doubling property at any smaller scale. This also yields a uniform neighbourhood in which $u$ enjoys slightly weaker doubling property. 

\begin{lemma}\label{lemma:double-unif}
Given any $d\in\N$ and $\gamma\in(0,\frac{1}{4})$, there exists a constant $\rho_0\in(0,\frac{1}{32})$, depending only on $n$, $\lambda$, $\omega$, $d$ and $\gamma$, such that if $a_+,a_-\in (C^{0,\omega}(B_1))^{n^2}$ satisfy \eqref{eq:a-ellip}, \eqref{eq:a-Dini} and \eqref{eq:a-k}, and if $u\in W^{1,2}(B_1)$ is a weak solution to \eqref{eq:main-sing} such that 
\begin{equation}\label{eq:double-asmp}
\norm{u}_{L^2(B_{16\rho})} \leq 2^{d + \gamma + \frac{n}{2}} \norm{u}_{L^2(B_{8\rho})},
\end{equation} 
for some $\rho\in (0,\rho_0]$, then 
\begin{equation}\label{eq:double-unif}
\sup_{z\in B_\rho} \sup_{r\in(0,4\rho]} \frac{\norm{u}_{L^2(B_{2r}(z))}}{\norm{u}_{L^2(B_r(z))}} \leq 2^{c_1d + c_2 + \frac{n}{2}}, 
\end{equation}
where $c_1$ is an absolute constant, while $c_2$ depends only on $n$ and $\lambda$.
\end{lemma} 

\begin{proof}
For the matter of simplicity, let us assume that $a_+ (0) = I$, which along with \eqref{eq:a-k} implies $a_-(0) = \kappa(0) I$ as well. Let us also denote $v_0 = u^+ - \kappa(0) u^-$ simply by $v$. Due to \eqref{eq:a-ellip} and \eqref{eq:a-k}, we have $\kappa(0) \in [\lambda^2,\lambda^{-2}]$, we obtain from \eqref{eq:double-asmp} that 
\begin{equation}\label{eq:v-double-1}
\norm{v}_{L^2(B_{16\rho})} \leq \frac{ 2^{d + \gamma + \frac{n}{2}} }{\lambda^4} \norm{v}_{L^2(B_{r_1})} \leq 2^{d + m_0 + \gamma + \frac{n}{2}}\norm{v}_{L^2(B_{8\rho})},
\end{equation}
where $m_0 = \min\{l\in \N: l \geq - \frac{4\log \lambda}{\log 2}\}$. 

Now fix $\rho_0$ as a small constant satisfying $\omega(16\rho_1) \leq \delta$, where $\delta$ is selected as in Lemma \ref{lemma:double} with $d$ replaced by $d + m_0$ and with the given $\gamma$. Clearly, $\rho_0$ depends on $n$, $\lambda$, $\omega$, $d$ and $\gamma$ only. Henceforth, we assume that $\rho \in (0,\rho_1]$. 

Given any integer $k\in\N\cup\{0\}$, let $u_k,v_k : B_1\ra\R$ be defined by $u_k (x) = u(2^{4-k}\rho x)$ and $v_k (x)= v(2^{4-k}\rho x) = u_k^+(x) - \kappa(0) u_k^-(x)$. In view of \eqref{eq:main-sing}, $u_k\in W^{1,2}(B_1)$ is a weak solution to $\ddiv(a_{k,+} \nabla u_k^+ - a_{k,-} \nabla u_k^- ) = 0$, with $a_{k,\pm} : B_1 \ra \R^{n^2}$ defined by $a_{k,\pm} (x) = a_\pm(2^{4-k}\rho x)$. It is evident that $a_{k,+},a_{k,-} \in (C^{0,\omega}(B_1))^{n^2}$ satisfy \eqref{eq:a-ellip}. Moreover, it follows immediately from \eqref{eq:a-Dini}, \eqref{eq:a-k} and the choice of $\rho_0$ that $\norm{a_{k,+} - I}_{L^\infty(B_1)} \leq \omega(2^{4-k}\rho) \leq \omega(16\rho_0) \leq\delta$, and similarly $\norm{a_{k,-} - \kappa(0)I }_{L^\infty(B_1)} \leq \delta$. Thus, Lemma \ref{lemma:double} implies that if 
\begin{equation}\label{eq:vk-double-1}
\norm{v_k}_{L^2(B_1)} \leq 2^{d + m_0 + \gamma + \frac{n}{2}} \norm{v_k}_{L^2(B_{1/2})},
\end{equation}
then
\begin{equation}\label{eq:vk-double-2}
\norm{v_{k+1}}_{L^2(B_1)} \leq 2^{d + m_0 + \gamma + \frac{n}{2}} \norm{v_{k+1}}_{L^2(B_{1/2})}. 
\end{equation}  
However, since  \eqref{eq:v-double-1} verifies \eqref{eq:vk-double-1} for $k=0$, we can iterate the implication from \eqref{eq:vk-double-1} to \eqref{eq:vk-double-2} for any integer $k\geq 0$, and arrive at
\begin{equation}\label{eq:v-double-2}
\sup_{r\in(0,8\rho]} \frac{\norm{v}_{L^2(2r)}}{\norm{v}_{L^2(B_r)}} \leq 2^{2(d + m_0 + \gamma + \frac{n}{2})}.
\end{equation} 

Now let $z\in B_\rho$ be arbitrary, and denote by $v_z$ the function $u^+ - \kappa(z) u^-$. Note that $B_{8\rho}(z) \subset B_{16\rho}$ and $B_{2\rho}\subset B_{4\rho}(z)$. Since \eqref{eq:a-ellip}, \eqref{eq:a-Dini} and \eqref{eq:a-k} together imply $|\kappa(z) - \kappa(0)| \leq c_0 \omega(\rho) \leq c_0\delta$, for some $c_0$ depending only on $\lambda$, we have $|v_z| \leq (1 + c_0\delta) |v_0|$ and $|v_0| \leq (1 + c_0\delta)|v_z|$ a.e. in $B_1$, for any $z\in B_\rho$. This together with \eqref{eq:v-double-2} and the above set inclusions, we deduce that 
\begin{equation}\label{eq:vz-double}
\begin{split}
\norm{v_z}_{L^2(B_{8\rho}(z))} &\leq \norm{v_z}_{L^2(B_{16\rho})} \\
&\leq (1+c_0\delta)^2 2^{6(d+m+\gamma + \frac{n}{2})}\norm{v_z}_{L^2(B_{2\rho})} \\
&\leq 2^{6d+m_1+\gamma + \frac{n}{2}} \norm{v_z}_{L^2(B_{4\rho}(z))},
\end{split}
\end{equation} 
where $m_1 = \min\{l\in\N: l\geq 6m + \frac{9}{4} + \frac{5n}{2}\} > 6m + 5\gamma + 1 + \frac{5n}{2}$; in the derivation of the second inequality, we took $\delta$ slightly smaller, if necessary, so that $(1+c_0\delta)^2\leq 2$. 

To this end, we choose $\delta$ even smaller so that Lemma \ref{lemma:double} holds with $d$ replaced by $6d + m_1$ and $\gamma = \frac{1}{8}$. Since $\omega(r_1)\leq \delta$, we can repeat the above proof that \eqref{eq:v-double-1} implies \eqref{eq:v-double-2}. Rephrasing it in terms of $u$ instead of $v_z$ by utilising the ellipticity bound $\kappa(z) \in [\lambda^2,\lambda^{-2}]$ again, we arrive at \eqref{eq:double-unif} with $c_1 = 12$ and $c_2 = m_1 + m_0 + \frac{1}{4} + \frac{n}{2}$, where $m_0$ is the positive integer chosen as in \eqref{eq:v-double-1}; tracking down the dependence of $m_0$ and $m_1$, we see that $c_2$ depends only on $n$ and $\lambda$. 

For the general case where $a_+(0)$ is no longer an identity matrix, one can change the coordinate system, follow the above argument for the scaled solution, and scale it back. Due to the ellipticity of $a_+$, constant $c_1$ remains to be an absolute constant, and $c_2$ is again determined by $n$ and $\lambda$ only, under such a transformation.
\end{proof} 

Next, we shall observe what determines scale $\rho$ such that $u$ satisfies the doubling property \eqref{eq:double-asmp}.

\begin{lemma}\label{lemma:double-factor}
Given any $\e>0$, $\gamma\in(0,\frac{1}{4})$, $R\in(0,1]$ and $d\in\N$, there exists a constant $r\in (0,\frac{R}{4})$, depending only on $n$, $\lambda$, $\omega$, $d$, $\e$ and $R$, such that if $a_+,a_- \in (C^{0,\omega}(B_1))^{n^2}$ satisfy \eqref{eq:a-ellip}, \eqref{eq:a-Dini} and \eqref{eq:a-k}, and if $u\in W^{1,2}(B_1)$ is a nontrivial weak solution to \eqref{eq:main-sing} satisfying  
\begin{equation}\label{eq:sing-d-e}
\limsup_{\rho\ra 0} \frac{\norm{u}_{L^2(B_\rho)}}{\rho^{\frac{n}{2} + d}} \geq \e \norm{u}_{L^2(B_1)},
\end{equation}
then one has 
\begin{equation}\label{eq:double-factor}
\norm{u}_{L^2(B_{2r})} \leq 2^{d+ \gamma + \frac{n}{2}} \norm{u}_{L^2(B_r)}.
\end{equation} 
\end{lemma}

\begin{proof}
To simplify our exposition, let us consider the case $R = 1$ only. Suppose towards a contradiction that there is some $\e_0>0$, and for each $j=1,2,\cdots$, one can find a pair $(a_{j,+},a_{j,-})$ of matrix-valued mappings in $(C^{0,\omega}(B_1))^{n^2}$ satisfying \eqref{eq:a-ellip}, \eqref{eq:a-Dini} and \eqref{eq:a-k} (where $\lambda$ and $\omega$ are fixed uniformly but $\kappa_j$ may vary), and a weak solution $u_j \in W^{1,2}(B_1)$ to \eqref{eq:main-sing} with $a_{j,+}$ and $a_{j,-}$ being the coefficient matrices such that  
\begin{equation}\label{eq:uj-sing-d-e}
\norm{u_j}_{L^2(B_1)} = 1,\quad \limsup_{\rho\ra 0} \frac{\norm{u_j}_{L^2(B_\rho)}}{\rho^{\frac{n}{2}+d}} \geq \e_0, \quad\text{but}
\end{equation} 
\begin{equation}\label{eq:double-factor-false}
\inf_{r \in [\frac{1}{j},\frac{1}{4})} \frac{\norm{u_j}_{L^2(B_{2r})}}{\norm{u_j}_{L^2(B_r)}} > 2^{d + \gamma + \frac{n}{2}}.
\end{equation} 

Let $v_j$ denote the function $u_j^+ - \kappa_j(0) u_j^-$. Since $0\in\cS_d(u_j) \subset \Gamma_{d-1}(u_j)$, Lemma \ref{lemma:int-Cd} (along with the assumption $\norm{u_j}_{L^2(B_1)} = 1$) ensures the existence of a homogeneous $a_{j,+}(0)$-harmonic polynomial $P_j$ with degree $d$ such that
\begin{equation}\label{eq:vj-Pj-W12}
|D^d P_j| + \sup_{r\in(0,\frac{1}{2})} \frac{ \norm{v_j - P_j}_{L^2(B_r)} + r \norm{\nabla (v_j - P_j)}_{L^2(B_r)} }{r^{\frac{n}{2} + d}\omega_1(r)} \leq C_0, 
\end{equation} 
where $C_0$ depends only on $n$, $\lambda$, $\omega$ and $d$. In particular, \eqref{eq:uj-sing-d-e} implies in \eqref{eq:vj-Pj-W12}, along with inequality $|v_j| \geq \lambda^2 |u_j|$ a.e. on $B_1$, that
\begin{equation}\label{eq:DdPj-0}
|D^d P_j| \geq c_0\e_0,
\end{equation} 
where $c_0\in(0,1)$ depends at most on $n$, $\lambda$ and $d$. 

As $\{P_j\}_{j=1}^\infty$ being a sequence of homogeneous polynomials with the same degree and uniformly bounded on $B_1$, we can find a homogeneous polynomial $P_0$ of degree $d$ such that $P_j\ra P_0$ in $C^d(B_1)$ along a subsequence, which we shall continue to denote by $\{P_j\}_{j=1}^\infty$. Passing to the limit in \eqref{eq:DdPj-0}, we observe that $P_0$ is nontrivial, and 
\begin{equation}\label{eq:DdP0-0}
|D^d P_0| \geq c_0\e_0. 
\end{equation} 

On the other hand, we can deduce from \eqref{eq:vj-Pj-W12} and the boundedness of $\{P_j\}_{j=1}^\infty$ in $C^d(B_1)$ that $\{v_j\}_{j=1}^\infty$ is bounded in $W^{1,2}(B_{1/2})$, hence after extracting a further subsequence if necessary, one may also assume that $v_j \ra v_0$ strongly in $L^2(B_{1/2})$, for some $v_0 \in W^{1,2}(B_{1/2})$. In particular, noting that $\kappa_j(0) \in [\lambda^2,\lambda^{-2}]$ from \eqref{eq:a-ellip} and \eqref{eq:a-k}, there exists a constant $\kappa_0\in[\lambda^2,\lambda^{-2}]$ such that $\kappa_j(0) \ra \kappa_0$, along another subsequence, so $u_j \ra u_0$ strongly in $L^2(B_{1/2})$ with $u_0 = v_0^+ - \kappa_0^{-1} u_0^-$, or equivalently, $v_0 = u_0^+ - \kappa_0 u_0^-$. 

To this end, take $j\ra\infty$ in \eqref{eq:vj-Pj-W12}, with each $r\in(0,\frac{1}{2})$ fixed. Then we can deduce from the strong convergence of $v_j\ra v_0$ in $L^2(B_{1/2})$ and the uniform convergence of $P_j\ra P_0$ on $B_1$ to deduce that
\begin{equation*}
\sup_{r\in(0,\frac{1}{2})} \frac{\norm{v_0 - P_0}_{L^2(B_r)}}{r^{\frac{n}{2} + d}\omega_1(r)} \leq C_0.
\end{equation*} 
This together with \eqref{eq:DdP0-0} and inequality $|u_0| \geq \lambda^2 |v_0|$ a.e. on $B_{1/2}$ yields that
\begin{equation}\label{eq:v0-sing-d}
\limsup_{r\ra 0} \frac{\norm{u_0}_{L^2(B_r)}}{r^{\frac{n}{2} + d}} \geq \lambda^2 \left( \limsup_{r\ra 0} \frac{\norm{v_0}_{L^2(B_r)}}{r^{\frac{n}{2}+ d}} \right)\geq c_1\e_0 > 0. 
\end{equation} 
for some $c_1\in(0,1)$ depending at most on $n$, $\lambda$ and $d$. 

However, passing to the limit in \eqref{eq:double-factor-false}, with each $r\in(0,\frac{1}{4})$ fixed, and utilising the strong convergence of $u_j \ra u_0$ in $L^2(B_{1/2})$, we deduce that 
\begin{equation*}
\inf_{r\in(0,\frac{1}{4})}\frac{\norm{u_0}_{L^2(B_{2r})}}{\norm{u_0}_{L^2(B_r)}} \geq 2^{d + \gamma + \frac{n}{2}},
\end{equation*} 
so after an iteration argument, we arrive at
\begin{equation*}
\limsup_{r\ra 0} \frac{\norm{u_0}_{L^2(B_r)}}{r^{\frac{n}{2} + d + \gamma }} \leq c_2 \norm{u_0}_{L^2(B_{1/2})} < \infty,
\end{equation*} 
for some $c_2>1$ depending only on $n$ and $d$, a contradiction to \eqref{eq:v0-sing-d}. This finishes the proof. 
\end{proof}

Now we are in a position to prove our main result of this section.

\begin{proof}[Proof of Theorem \ref{theorem:double}]

After normalisation, we can assume without loss of any generality that $\norm{u}_{L^2(B_1)} = 1$. As $d$ being the maximal vanishing order of $u$, it follows from Theorem \ref{theorem:sing} that $\{u=0\}  = \bigcup_{k=1}^d \cS_d(u)$, where we wrote by $\cS_1(u)$ the regular part, $\cN(u)$ (see \eqref{eq:reg} for its definition), of the nodal set. Moreover, $\Gamma_{d-1}(u)  = \cS_d(u) $ (with $\Gamma_0(u) = \{u=0\}$), so there should exist a small constant $\e_1>0$ such that for any $z\in\cS_d(u)\cap\overline{B_{1/2}}$, 
\begin{equation}\label{eq:double-claim}
\limsup_{r\ra 0} \frac{\norm{u}_{L^2(B_r(z))}}{r^{\frac{n}{2} + d}} \geq \e_1.  
\end{equation} 

Suppose that our claim \eqref{eq:double-claim} does not hold. Then we can find a sequence $\{z_j\}_{j=1}^\infty\subset\cS_d(u)\cap \overline{B_{1/2}}$ such that 
\begin{equation*}
\limsup_{r\ra 0} \frac{\norm{u}_{L^2(B_r(z))}}{r^{\frac{n}{2} + d}} \leq \frac{1}{j},
\end{equation*}
so it follows from the uniform $C^d$-approximation (Lemma \ref{lemma:int-Cd}) that $P_{j,z_j}$, the homogeneous $a_+(z_j)$-harmonic polynomial of degree $d$ approximating $v_{z_j} = u^+ - \kappa(z_j) u^-$ at $z_j$, must satisfy 
\begin{equation}\label{eq:DdPzj-false}
|D^d P_{z_j}| \leq \frac{C_0}{j},
\end{equation} 
for some $C_0>0$ independent of $j$. Extracting a subsequence of $\{z_j\}_{j=1}^\infty$ along which $z_j\ra z_0$ for some $z_0 \in \overline{B_{1/2}}$, we know that $z_0 \in \Gamma_{d-1}(u)\cap \overline{B_{1/2}}$, since $\Gamma_{d-1}(u)$ is a relatively closed set in $B_1$, according to Lemma \ref{lemma:closed}. Therefore, we can apply Lemma \ref{lemma:int-Cd} again to $z_0$, and obtain another homogeneous $a_+(z_0)$-harmonic polynomial of degree $d$, approximating $v_{z_0} = u^+ - \kappa (z_0) u^-$ at $z_0$ in the sense that
\begin{equation}\label{eq:vz0-Pz0-W12}
\sup_{r\in(0,\frac{1}{4})} \frac{\norm{v_{z_0} - P_{z_0}}_{L^2(B_r(z_0))} }{r^{\frac{n}{2} + d} \omega_1(r)} \leq C_1, 
\end{equation}
where $C_1>0$ depends only on $n$, $\lambda$, $\omega$ and $d$. However, due to the uniform continuity \eqref{eq:DdPz-C} of the mapping $z\mapsto D^d P_z$ on $\Gamma_{d-1}(u)\cap \overline{B_{1/2}}$, we may infer from \eqref{eq:DdPzj-false} as well as the convergence $z_j\ra z_0$ that 
\begin{equation}\label{eq:DdPz0-0}
|D^d P_{z_0}| = 0.
\end{equation} 
However, owing to \eqref{eq:vz0-Pz0-W12}, \eqref{eq:DdPz0-0} implies that $z_0 \in \Gamma_d(u)$, a contradiction against our assumption that $u$ vanishes on $\{u=0\}$ with order at most $d$. This justifies our claim that there must exist some $\e_1>0$ such that \eqref{eq:double-claim} holds for all $z\in \cS_d(u)\cap\overline{B_{1/2}}$. 

With \eqref{eq:double-claim} at hand, we can apply Lemma \ref{lemma:double-factor} to all $z\in \cS_d(u)\cap \overline{B_{1/2}}$, with $\e = \e_1$, $\gamma = \frac{1}{8}$ (for the matter of simplicity) and $R = 8\rho_0$, where $\rho_0 \in (0,\frac{1}{64})$ is chosen as in Lemma \ref{lemma:double-unif} with the same $d$, $\gamma = \frac{1}{8}$ and domain $B_1$ replaced by $B_{1/2}$. Then it will give us a radius $r_1\in(0,8\rho_0)$, determined by $n$, $\lambda$, $\omega$, $d$ and $\e_1$ only, such that 
\begin{equation}\label{eq:double-z}
\sup_{z\in \cS_d(u)\cap \overline{B_{1/2}}} \frac{\norm{u}_{L^2(B_{2r_1}(z))}}{\norm{u}_{L^2(B_{r_1}(z))}} \leq 2^{d + \frac{1}{8} + \frac{n}{2}}.
\end{equation}
Thus, $u$ satisfies the doubling property \eqref{eq:double-asmp} at each $z\in\cS_d(u)\cap \overline{B_{1/2}}$, with $\rho_1 = \frac{r_1}{8} \in (0,\rho_0]$ and $\gamma = \frac{1}{8}$, so Lemma \ref{lemma:double-unif} ensures that with $N_1 = \bigcup_{x\in \cS_d(u)\cap \overline{B_{1/2}}} B_{\rho_1}(x)$, we have 
\begin{equation}\label{eq:double-nbd}
\sup_{z\in N_1} \sup_{r \in (0,4\rho_1]} \frac{\norm{u}_{L^2(B_{2r}(z))}}{\norm{u}_{L^2(B_r(z))}} \leq 2^{c_1 d + c_2 + \frac{n}{2}}, 
\end{equation}
where $c_1$ is an absolute constant and $c_2$ depends only on $n$ and $\lambda$.  

To summarise, we have obtained a neighbourhood $N_1$, of $\cS_d(u)\cap \overline{B_{1/2}}$ with size $\rho_1$, where $u$ satisfies a uniform doubling property with order $c_1 d + c_2$. 

To this end, we can repeat the above argument to obtain a uniform doubling property for the rest of the nodal set, $\{u=0\} \cap \overline{B_{1/2}} \setminus N_1$. One can proceed as follows. 

First, assume that we have found some $\rho_{d-k} \in (0,\frac{1}{64})$ for which \eqref{eq:double-nbd} holds with $N_1$ replaced by $N_{d-k} = \bigcup_{x\in \Gamma_k(u)\cap \overline{B_{1/2}}} B_{\rho_{d-k}}(x)$, and $\rho_1$ replaced by $\rho_{d-k}$. Then by the definition of $N_{d-k}$, we have $\dist(\cS_k(u)\cap\overline{B_{1/2}}\setminus N_{d-k}, \Gamma_k(u)\cap \overline{B_{1/2}})\geq \rho_{d-k}$. Therefore, we can argue as in the proof for our earlier claim \eqref{eq:double-claim} and obtain some $\e_{d-k+1}>0$, depending only on $n$, $\lambda$, $\omega$, $k$, $\rho_{d-k}$ and certain character of $u$, such that \eqref{eq:double-claim} holds with $d$ and $\e_1$ replaced by $k$ and $\e_{d-k+1}$ respectively. Then we can follow the argument for \eqref{eq:double-nbd} and obtain $\rho_{d-k+1}\in (0,\rho_{d-k}]$, depending only on $n$, $\lambda$, $\omega$, $k$, $\rho_{d-k}$ and $\e_{d-k+1}$, such that \eqref{eq:double-nbd} holds with $N_1$ replaced by $N'_{d - k +1} = \bigcup_{x\in \cS_k(u)\cap \overline{B_{1/2}}\setminus N_{d-k}} B_{\rho_{d-k+1}}(x)$, and $\rho_1$ replaced by $\rho_{d-k+1}$. This combined with the initial hypothesis on the existence of $\rho_{d-k}$, as well as the observation that $\Gamma_{k-1}(u) = \cS_k(u)\cup \Gamma_k(u)$ and $\rho_{d-k+1} \leq \rho_{d-k}$, implies that \eqref{eq:double-nbd} holds with $N_1$ replaced by $N_{d-k+1} = \bigcup_{x\in\Gamma_k(u)\cap \overline{B_{1/2}}} B_{\rho_{d-k+1}}(x)$ and $\rho_1$ replaced by $\rho_{d-k+1}$. Note that constants $c_1$ and $c_2$ in \eqref{eq:double-nbd} do not change during this step. 

Iterating this argument for $k$ running from $d-1$ to $1$, which only consists of a finite number of steps, we arrive at the desired conclusion, with $r_0$ in \eqref{eq:double} given by $\rho_d \in (0,\rho_{d-1}]\subset\cdots\subset (0,\rho_0]$. Tracking down the dependence on each $\rho_k$, we conclude that $r_0$ is determined solely by $n$, $\lambda$, $\omega$, $d$ and certain character of $u$. This finishes the proof. 
\end{proof}


\section{Measure Estimate of Nodal Set}\label{section:nodal}

In this section, we shall estimate the $(n-1)$-dimensional Hausdorff measure of the nodal set of weak solutions to a broken quasilinear PDE. Our analysis will follow closely the classical works \cite{HS}, by Hardt and Simon, and \cite{HL94}, by Han and Lin, and it can be considered as an extension of these works to the setting of broken quasilinear PDEs. 

\begin{theorem}\label{theorem:meas}
Suppose that $a_+,a_- \in (C^{0,\omega}(B_1))^{n^2}$ satisfy \eqref{eq:a-ellip}, \eqref{eq:a-Dini} and \eqref{eq:a-k}, and let $u\in W^{1,2}(B_1)$ be a weak solution to \eqref{eq:main-sing}. Given any $N\geq 1$, there exists a constant $\delta_N > 0$, depending only on $n$, $\lambda$, $\omega$ and $N$, such that if $\omega(1)\leq \delta_N$, and 
\begin{equation}\label{eq:almgren}
\sup_{z\in B_1} \sup_{r\in(0,1-|z|)} \frac{ r\int_{B_r(z)} |\nabla u|^2\,dx }{\int_{\partial B_r(z)} u^2\,d\sigma_x} \leq N, 
\end{equation} 
then one has 
\begin{equation}\label{eq:meas}
H^{n-1} ( \{u = 0\} \cap B_{1/2}) \leq CN,
\end{equation} 
where $C>0$ depends only on $n$, $\lambda$ and $\omega$. 
\end{theorem} 

Due to the uniform doubling property (Theorem \ref{theorem:double}) for solutions featuring finite vanishing order, we obtain a global measure estimate as a corollary to the above theorem.

\begin{corollary}\label{corollary:meas}
Under the setting of Theorem \ref{theorem:double}, there exists a positive constant $C_{d,u}$, depending only on $n$, $\lambda$, $\omega$, $d$ and certain character of $u$, such that
\begin{equation*}
H^{n-1} (\{ u =0 \} \cap B_{1/2}) \leq C_{d,u}. 
\end{equation*} 
\end{corollary}

\begin{proof}
The uniform doubling property \eqref{eq:double} combined with a trace inequality $\int_{B_r} w^2\,dx \leq \frac{1}{n}(r + \frac{1}{\e}) \int_{\partial B_r} w^2\,dx + \frac{\e r^2}{n} \int_{B_r} |\nabla w|^2\,dx$, which holds for any $w\in W^{1,2}(B_r)$ and any $\e>0$, implies that
\begin{equation}\label{eq:double-meas}
\sup_{z\in \cB_{d,u}} \sup_{r\in (0,r_{d,u}]} \frac{ r\int_{B_r(z)} |\nabla u|^2\,dx }{\int_{\partial B_r(z)} u^2\,d\sigma_x} \leq c_1d + c_2,
\end{equation} 
for some absolute constant $c_1>1$ and a positive constant $c_2$ depending only on $n$ and $\lambda$, where $r_{d,u}$ and $\cB_{d,u}$ are as in Theorem \ref{theorem:double}. Take $N = c_1d + c_2$, and choose a small $r_N < r_{d,u}$ such that $\omega(r_N) \leq \delta_N$, where $\delta_N$ is given as in Theorem \ref{theorem:meas}. Clearly, $r_N$ is determined by $n$, $\lambda$, $\omega$, $N$ and $r_{d,u}$. Due to \eqref{eq:double-meas}, one can apply Theorem \ref{theorem:meas} for any $z\in \{u=0\}\cap B_{1/2}$, after some scaling argument, and deduce that
\begin{equation*}
H^{n-1}\left( \{ u = 0\} \cap B_{\frac{r_N}{2}} (z) \right) \leq CNr_N^{n-1}.  
\end{equation*}
Hence, a standard covering argument yields the conclusion. 
\end{proof} 

As mentioned briefly in the beginning of this section, our argument will follow the classical works \cite{HS} and \cite{HL94}, which mainly consists of the following three steps: (i) first, to compare the given solution to a harmonic function having comparable frequency; (ii) second, to compare the regular part of the nodal set of the solution and that of the approximating harmonic function, which is called the nodal set comparison lemma (c.f. Lemma 4.8 in \cite{HS}); (iii) third, to cover the singular part by a finitely many balls with sufficiently small radii. 

Throughout all these steps, the $C^{1,\alpha}$-regularity of the given solution and the closeness (or compactness) in $C^1$-norm play an essential role in the aforementioned articles. Replacing $C^{1,\alpha}$-regularity of solutions to $C^1$-regularity does not require much work, so it is relatively straightforward to extend the argument to uniformly elliptic, divergence-type PDEs with Dini coefficients (without jump discontinuity); still, such an extension has not been considered yet, to the best of our knowledge. 

What seems to be more nontrivial, which is exactly the place that we would like to claim the novelty of this section for, is to extend the argument to broken quasilinear PDEs, since then the solutions are Lipschitz regular at their best, so one cannot hope for closeness in $C^1$-norm. However, due to Theorem \ref{theorem:int-Lip}, $\nabla u^+ - \kappa \nabla u^-$ has a uniformly continuous version, $Vu$, across the nodal set, and it turns out that if $u^+ - \kappa u^-$ approximates a harmonic function, then $Vu$ approximates its gradient, with respect to the supremum norm. Hence, we can substitute with the very mapping $Vu$ the role of the usual gradient in the classical argument. Now with the particular structure that $u^+ - \kappa u^-$ has the same nodal set with $u$, we can proceed with the comparison of the nodal set as well. 

Henceforth, we shall focus on the part where either we need to involve a new argument adapted to our setting, or the modification is not straightforward. Let us begin with the compactness lemma.

\begin{lemma}\label{lemma:compact}
Suppose that $a_+,a_-\in (C^{0,\omega}(B_1))^{n^2}$ satisfy \eqref{eq:a-ellip}, \eqref{eq:a-Dini} and \eqref{eq:a-k}, and let $u\in W^{1,2}(B_1)$ be a weak solution to \eqref{eq:main-sing}.
\begin{enumerate}[(i)]
\item If $\norm{a_\pm - a_\pm(0)}_{L^\infty(B_1)} \leq \delta$, then with the $a_+(0)$-harmonic function $h \in (u^+ - \kappa(0) u^-) + W_0^{1,2}(B_1)$, one has 
\begin{equation}\label{eq:u-h-Linf}
\sup_{B_{3/4}} | u^+ - \kappa u^- - h | \leq C\delta \norm{u}_{L^2(B_1)},
\end{equation}
where $C>0$ depends only on $n$, $\lambda$ and $\omega$. 
\item Moreover, if $\delta\in(0,\frac{1}{2}]$ in (i), and if there is some $N \geq 1$ that 
\begin{equation}\label{eq:u-freq}
\frac{\int_{B_1} |\nabla u|^2\,dx}{\int_{\partial B_1} u^2\,d\sigma_x} \leq N, 
\end{equation}
then with the same $h$ as in (i), we have
\begin{equation}\label{eq:h-freq}
\frac{\int_{B_1} |\nabla h|^2\,dx}{\int_{\partial B_1} h^2\,d\sigma_x} \leq c N,
\end{equation} 
where $c>1$ depends only on $\lambda$. 
\item On the other hand, for any given $\e>0$, one can find some $\bar\delta \in(0,\frac{1}{2})$, depending only on $n$, $\lambda$, $\omega$ and $\e$, such that if $\delta \in (0,\bar\delta]$ in (i), then with $Vu$ given as in Theorem \ref{theorem:int-Lip}, and with the same $h$ as in (i), we have 
\begin{equation}\label{eq:Du-Dh-Linf}
\sup_{B_{3/4}} | Vu - \nabla h| \leq \e \norm{u}_{L^2(B_1)}.
\end{equation} 
\end{enumerate}
\end{lemma} 

\begin{proof}
Without loss of generality, let us assume that $\norm{u}_{L^2(B_1)} = 1$, since the statement becomes trivial for a solution vanishing everywhere. 

Let us begin with the proof of the first assertion. Denoting by $v$ function $u^+ - \kappa(0) u^-$, we observe that $h - v \in W_0^{1,2}(B_1)$ is a weak solution to 
\begin{equation}\label{eq:h-v-pde}
\ddiv( a_+(0)\nabla w) = \ddiv ((a_+ - a_+(0))\nabla u^+ - (a_- -a_-(0))\nabla u^-).
\end{equation}  
Hence, we deduce from the closeness condition on $a_\pm - a_\pm(0)$, as well as the ellipticity \eqref{eq:a-ellip} of $a_+$ and $a_-$, that 
\begin{equation}\label{eq:Dh-Dv-L2}
\norm{ \nabla (h-v)}_{L^2(B_1)} \leq \frac{\delta}{\lambda}\norm{\nabla u}_{L^2(B_1)} \leq C_1\delta,
\end{equation} 
where the last inequality follows from the local energy estimate of $u$, with $C_1>0$ depending only on $n$ and $\lambda$. Since $h-v \in W_0^{1,2}(B_1)$, the Poincar\'e inequality yields that
\begin{equation}\label{eq:h-v-L2}
\norm{h-v}_{L^2(B_1)} \leq c_1C_1\delta, 
\end{equation}
for some $c_1>0$ depending only on $n$. 

However, since Theorem \ref{theorem:int-Lip} implies that 
\begin{equation}\label{eq:u-W1inf}
\norm{u}_{W^{1,\infty}(B_{7/8})} \leq C_2,
\end{equation}
for some $C_2>0$ depending only on $n$, $\lambda$ and $\omega$, we can employ the local $L^\infty$-estimate to \eqref{eq:h-v-pde} and obtain from \eqref{eq:h-v-L2}, \eqref{eq:u-W1inf} and the closeness condition on $a_\pm-a_\pm(0)$ that
\begin{equation}\label{eq:h-v-Linf}
\norm{h-v}_{L^\infty(B_{3/4})} \leq C_3\delta. 
\end{equation} 

On the other hand, one can easily deduce from $\norm{a_\pm - a_\pm(0)}_{L^\infty(B_1)} \leq \delta$, as well as \eqref{eq:a-ellip} and \eqref{eq:a-k}, that $\norm{\kappa - \kappa(0)}_{L^\infty(B_1)} \leq c_1\delta$, for some $c_1>0$ depending at most on $n$. This combined with \eqref{eq:h-v-Linf} as well as \eqref{eq:u-W1inf} implies that 
\begin{equation}\label{eq:h-v-Linf-re}
\begin{split}
\norm{u^+ - \kappa u^- - h}_{L^\infty(B_{3/4})} &\leq \norm{u^+ - \kappa u^- - v}_{L^\infty(B_{3/4})} + C_3\delta \\
&\leq \delta \left( c_1 \norm{u^-}_{L^\infty(B_{3/4})} + C_3\right) \\
&\leq (c_1 C_2 + C_3)\delta. 
\end{split} 
\end{equation} 
This proves \eqref{eq:h-v-Linf} with $C = c_1C_2 + C_3$. 

Next, let us prove the second assertion. Due to the standing assumptions \eqref{eq:a-ellip} and \eqref{eq:a-k} that ensure $\lambda^2\leq \kappa(0)\leq \lambda^{-2} $, we have 
\begin{equation*}
\int_{\partial B_1} h^2\,d\sigma_x = \int_{\partial B_1} \left( (u^+)^2 + \kappa(0)^2 (u^-)^2\right)\,d\sigma_x \geq \lambda^4\int_{\partial B_1} u^2\,d\sigma_x.
\end{equation*}
On the other hand, in view of \eqref{eq:h-v-pde} and assumption $\delta < \frac{1}{2}$, it follows immediately from \eqref{eq:a-ellip} and \eqref{eq:u-freq} that
\begin{equation*}
\int_{B_1} |\nabla (h - (u^+ - \kappa(0)u^-))|^2\,dx  \leq \frac{1}{4\lambda^4} \int_{B_1} |\nabla u|^2\,dx.
\end{equation*} 
Combining the last two estimates, we arrive at \eqref{eq:h-freq}. 

We are now in a position to verify the last assertion. Suppose towards a contradiction that for each $j=1,2,\cdots$, we can find a pair $(a_{j,+},a_{j,-})$ of coefficient matrices in $(C^{0,\omega}(B_1))^{n^2}$ satisfying \eqref{eq:a-ellip}, \eqref{eq:a-Dini}, and \eqref{eq:a-k} as well as $\norm{a_{j,\pm} - a_{j,\pm}(0)}_{L^\infty(B_1)} \leq \frac{1}{j}$, and a nontrivial weak solution $u_j\in W^{1,2}(B_1)$ to \eqref{eq:main-sing} with $\norm{u_j}_{L^2(B_1)} = 1$, such that with the $a_+(0)$-harmonic function $h_j \in  v_j + W_0^{1,2}(B_1)$, with $v_j = u_j^+ - \kappa_j(0) u_j^-$, one has 
\begin{equation}\label{eq:Duj-Dhj-Linf}
\sup_{B_{3/4}} | V_ju_j - \nabla h_j| > \e_0,
\end{equation} 
infinitely often as $j\ra\infty$, with a fixed constant $\e_0>0$, where $V_ju_j$ is given as in Theorem \ref{theorem:int-Lip}; here we specify the subscript $j$ on $V_j$ as the coefficient matrices $a_{j,+},a_{j,-}$ are varying along the sequence.

Following the above argument, we can verify that $h_j$ satisfies \eqref{eq:h-v-Linf-re}, which now reads
\begin{equation}\label{eq:uj-hj-Linf}
\norm{u_j^+ - \kappa_j u_j^- - h_j}_{L^\infty(B_{7/8})} \leq \frac{C_4}{j}, 
\end{equation} 
for some $C_4>0$ depending only on $n$, $\lambda$ and $\omega$. 

Without loss of generality, we may assume that $a_{j,+}(0) \ra I$, and $\kappa_{j,+}(0) \ra \kappa_0$, for some constant $\kappa_0\in[\lambda^2,\frac{1}{\lambda^2}]$, as $j\ra\infty$. Then due to the assumption that $\norm{a_{j,\pm} - a_{j,\pm}(0)}_{L^\infty(B_1)} \leq \frac{1}{j}$, we have $a_{j,+} \ra I$, $a_{j,-} \ra \kappa_0 I$ and $\kappa_j \ra \kappa_0$ uniformly on $B_1$. 

Due to \eqref{eq:h-v-L2} (applied to $v_j$ and $h_j$) and the local energy estimate on $u_j$, we know that $\norm{h_j}_{L^2(B_1)} \leq C_5$ for some $C_5>0$ depending only on $n$ and $\lambda$. Hence, it follows from the interior estimates for harmonic functions that $\norm{h_j}_{C^{1,1}(\overline{B_{7/8}})}\leq C_6$, for some $C_6>0$ depending only on $n$ and $\lambda$. Thus, there is some $h_0 \in C^{1,1}(\overline{B_{7/8})}$ such that $h_j \ra h_0$ and $\nabla h_j\ra \nabla h_0$ uniformly on $\overline{B_{7/8}}$, along a subsequence that we will continue to denote by $\{h_j\}_{j=1}^\infty$. Since we have $a_{j,+}(0) \ra I$, the uniform convergence of $\{h_j\}_{j=1}^\infty$ to $h_0$ implies that $h_0$ is a harmonic function on $B_{3/4}$. 

Moreover, since $u_j$ also verifies \eqref{eq:u-W1inf}, one can find some $u_0\in W^{1,\infty}(B_{7/8})$ such that $u_j\ra u_0$ uniformly over $\overline{B_{7/8}}$ and $\nabla u_j \ra \nabla u_0$ strongly in $L^2(B_{7/8})$, after extracting a subsequence if necessary. Recalling that $\kappa_j\ra \kappa_0$ uniformly on $B_1$, for certain constant $\kappa_0$, we can observe, by passing to the limit in \eqref{eq:uj-hj-Linf}, that 
\begin{equation}\label{eq:u0-h0}
h_0 = u_0^+ -\kappa_0 u_0^-\quad\text{in }B_{7/8}. 
\end{equation} 

In addition, due to Theorem \ref{theorem:int-Lip}, or \eqref{eq:int-Lip} to be more specific, we also know that $\norm{V_ju_j}_{C^{0,\omega_1}(\overline{B_{3/4}})} \leq C_7$ for some $C_7>0$ depending only on $n$, $\lambda$ and $\omega$. This yields another mapping $p_0 \in (C^{0,\omega_1}(\overline{B_{3/4}})^n$ to which $\{V_ju_j\}_{j=1}^\infty$ converge uniformly on $\overline{B_{3/4}}$, after extracting a further subsequence if necessary. 

On the other hand, applying \eqref{eq:int-C1-re} (which also holds for $p = 2$) to $u_j$ and $V_ju_j$, we have
\begin{equation}\label{eq:Duj-pj-L2}
\sup_{z\in\overline{B_{3/4}}} \sup_{r\in(0,\frac{1}{8})} \frac{\norm{\nabla u_j^+ - \kappa_j(z) \nabla u_j^- - V_ju_j(z)}_{L^2(B_r(z))}}{r^{\frac{n}{2}} \omega_1(r)} \leq C_9,
\end{equation} 
for some constant $C_9>0$ that depends only on $n$, $\lambda$ and $\omega$. Since the boundedness of $\{u_j\}_{j=1}^\infty$ in $W^{1,\infty}(B_{7/8})$ implies the strong convergence $\nabla u_j \ra \nabla u_0$ in $L^2(B_{7/8})$, and since \eqref{eq:u0-h0} implies $\nabla h_0 = \nabla u_0^+ -\kappa_0 \nabla u_0^-$ a.e. in $B_{7/8}$, taking limits in \eqref{eq:Duj-pj-L2} with each $r\in(0,\frac{1}{8})$ fixed, we deduce that 
\begin{equation*}
\sup_{z\in \overline{B_{3/4}}} \sup_{r\in(0,\frac{1}{8})} \frac{\norm{\nabla h_0 - p_0(z)}_{L^2(B_r(z))}}{r^{\frac{n}{2}} \omega_1(r)} \leq C_9,
\end{equation*} 
and in particular, 
\begin{equation}\label{eq:Dh0-p0}
\nabla h_0 = p_0\quad\text{in }\overline{B_{3/4}}.
\end{equation} 
This leads us to a contradiction to \eqref{eq:Duj-Dhj-Linf}, since $\nabla h_j \ra \nabla h_0$ and $V_ju_j\ra p_0$ uniformly on $\overline{B_{3/4}}$. This finishes the proof. 
\end{proof}

Next we shall present a nodal set comparison lemma for weak solutions to broken quasilinear PDEs \eqref{eq:main-sing}. 

\begin{lemma}\label{lemma:compare}
Let $a_+,a_- \in (C^{0,\omega}(B_1))^{n^2}$ satisfy \eqref{eq:a-ellip}, \eqref{eq:a-Dini} and \eqref{eq:a-k}, and suppose that $u \in W^{1,2}(B_1)$ is a nontrivial weak solution to \eqref{eq:main-sing} such that $\norm{u}_{L^2(B_1)} \leq 1$, with mapping $Vu$ as in Theorem \ref{theorem:int-Lip}. Then given any real $\eta \in (0,\frac{1}{16})$, and modulus of continuity $\theta$, there exists a constant $\e>0$, depending only on $n$, $\lambda$, $\omega$, $\theta$ and $\eta$, such that if $h \in C^{1,\theta}(B_{7/8})$ is an arbitrary function for which $\norm{h}_{C^{1,\theta}(B_{7/8})} \leq 1$, $\sup_{B_{3/4}} |u^+ - \kappa u^- - h| \leq \e$ and $\sup_{B_{3/4}} |Vu - \nabla h| \leq \e$, then 
\begin{equation}\label{eq:compare}
\begin{split}
&H^{n-1} \left( B_{\frac{3}{4} - \eta} \cap \{ u = 0, |Vu| \geq \eta\} \right) \\
&\leq (1 + C\omega_1(\eta)) H^{n-1} \left( B_{\frac{3}{4}} \cap \left\{ h = 0, |\nabla h| \geq\frac{\eta}{2}\right\} \right),
\end{split}
\end{equation} 
where $C>0$ is a constant depending only on $n$, $\lambda$, $\omega$ and $\theta$. 
\end{lemma} 

\begin{proof}
The main idea of the proof is the same with Lemma 4.8 in \cite{HS}, so let us present a sketch of our argument.

Fix a point $z\in \{ u = 0 , |Vu|\geq \eta\} \cap B_{\frac{3}{4} - \eta}$. Due to Theorem \ref{theorem:fb-reg}, one can choose a small $\rho \in (0,\frac{1}{8})$, depending only on $n$, $\lambda$, $\omega$ and $\eta$, such that $\{ u = 0 \}\cap B_\rho(z)$ is a $C^{1,\omega_1}$-graph. In particular, denoting by $\Pi_z$ the tangent hyperplane to this graph, we can find some function $\psi_z \in C^{1,\omega_1}(B_\rho(z))$ such that
\begin{equation}\label{eq:comp-0}
\{u = 0\}\cap B_\rho(z) = \left\{ y + \psi_z (y) \frac{Vu(z)}{|Vu(z)|} : y \in \Pi_z \right\}.
\end{equation}
Now taking $\rho$ smaller if necessary, one can also obtain
\begin{equation}\label{eq:comp-01}
\sup_{B_\rho(z)} | \nabla \psi_z | \leq \omega_1(\eta)
\end{equation}
This can be done with Theorem \ref{theorem:int-Lip}. Since $\norm{u}_{L^2(B_1)} \leq 1$, we deduce that $\norm{u}_{W^{1,\infty}(B_{7/8})} + \norm{Vu}_{C^{0,\omega_1}(B_{7/8})} \leq C_0$, for some $C_0>0$ depending only on $n$, $\lambda$ and $\omega$. Since $|Vu(z)| \geq \eta$, we can choose a smaller $\rho$, but still depending on $n$, $\lambda$, $\omega$ and $\eta$, such that 
\begin{equation}\label{eq:comp-1}
\inf_{B_\rho(z)} |Vu| \geq \frac{4\eta}{5},\quad \text{and}
\end{equation}
\begin{equation}\label{eq:comp-11}
\max_{1\leq i\leq n} \osc_{B_\rho(z)} \frac{(Vu)\cdot e_i }{|Vu|}\leq \omega_1(\eta).
\end{equation}

Now since $\sup_{B_{3/4}} |u^+ - \kappa u^- - h| \leq \e$ and $\sup_{B_{3/4}} |Vu - \nabla h| \leq \e$, we can select $\e>0$, depending only on $\rho$, $\eta$ and $\omega$, to deduce from $z\in\{u = 0\}\cap B_{\frac{3}{4} - \eta}$, \eqref{eq:comp-1} and \eqref{eq:comp-11} that 
\begin{equation}\label{eq:comp-2}
\{h = 0\}\cap B_{\rho\omega_1(\eta)} \neq \emptyset.
\end{equation} 
\begin{equation}\label{eq:comp-21}
\inf_{B_\rho(z)} |h| \geq \frac{3\eta}{4},\quad\text{and} 
\end{equation} 
\begin{equation}\label{eq:comp-22}
\sup_{B_\rho(z)} \left| \frac{Vu}{|Vu|} - \frac{\nabla h}{|\nabla h|} \right| \leq \rho\omega_1(\eta). 
\end{equation} 

Choosing $\rho$ even smaller, now depending on $\theta$ (the $C^1$-character of $h$) but no more, $\{h = 0\}\cap B_\rho(z)$ becomes a $C^{1,\theta}$-graph because of \eqref{eq:comp-21}. Moreover, due to \eqref{eq:comp-2} and \eqref{eq:comp-22}, we can express
\begin{equation}
\{h = 0\} \cap B_\rho(z) = \left\{ y + \zeta_z (y) \frac{Vu(z)}{|Vu(z)|} : y \in \Pi_z \right\},
\end{equation} 
(even if $z\neq \{h = 0\}\cap B_\rho(z)$, and $\Pi_z$ is not a perpendicular to $|h(z)|^{-1} h(z)$), with some real-valued function $\zeta_z\in C^{1,\theta}(B_\rho(z))$ satisfying
\begin{equation}\label{eq:comp-24}
\sup_{B_\rho(z)} |\nabla \zeta_z| \leq 2\omega_1(\eta),
\end{equation}
because of \eqref{eq:comp-11} and \eqref{eq:comp-22}. Tracking down the dependence of $\rho$ and $\e$, we observe that $\rho$ depends only on $n$, $\lambda$, $\omega$, $\theta$ and $\eta$, and so does $\e$. 

Finally, we can employ the covering argument as in the proof of Lemma 4.8 in \cite{HS} to derive \eqref{eq:compare}. This final step only involves some standard techniques in analysis, so we shall refer to \cite{HS} for the details. 
\end{proof} 

Finally, we are ready to estimate the $H^{n-1}$-measure of the nodal set. We shall follow the argument of Theorem 3.1 in \cite{HL94}, where the main tools, namely the compactness and the comparison lemma for nodal sets, are already established in Lemma \ref{lemma:compact} and respectively Lemma \ref{lemma:compare}. The rest of the argument is more or less straightforward, so we shall only sketch the proof and refer to \cite{HL94} for the details. 

\begin{proof}[Proof of Theorem \ref{theorem:meas}]
As in the proof of Theorem 3.1 in \cite{HL94}, we are going to construct a sequence $\{\cB_k\}_{k=0}^\infty$ of finite covers of $\{u=0\}\cap B_{1/2}$ such that for each $k\in\N\cup\{0\}$, any ball $B\in \cB_k$ has radius not exceeding $2^{-k}$, so that 
\begin{equation}\label{eq:Bk-1}
\{u= 0, |Vu| = 0\}\cap B_{1/2} \subset \bigcup_{B \in \cB_k} B, 
\end{equation}
\begin{equation}\label{eq:Bk-2}
\sum_{B \in \cB_k} \left(\frac{\diam B}{2}\right)^{n-1} \leq \frac{1}{2^k}, \quad\text{and}
\end{equation} 
\begin{equation}\label{eq:Bk-3}
H^{n-1} \left( \{u=0\}\cap \bigcup_{B\in \cB_{k-1}} B\setminus \bigcup_{B\in \cB_k} B \right) \leq \frac{cN}{2^k},
\end{equation} 
where $c$ is a positive constant depending only on $n$ and $\lambda$. Our conclusion will follow immediately from the way these coverings are constructed. 

Let $\cB_0 = \{B_{1/2}\}$, and assume that we have already constructed $\cB_k$, for $k=0,1,\cdots,l-1$, for some $l\in \N$. To construct $\cB_l$, let us fix $B = B_r(z) \in \cB_{l-1}$. Consider the scaling function $v: \overline{B_1} \ra \R$ defined by 
\begin{equation*}
v(x) =  \frac{u(2rx + z)}{(2r)^{\frac{1-n}{2}}\norm{u}_{L^2(B_{2r}(z))}}.
\end{equation*}
Then $v \in W^{1,2}(B_1)$ and it satisfies \eqref{eq:main-sing} in the weak sense with coefficients $b_+,b_- : B_1 \ra\R^{n^2}$ given by $b_\pm (x) = a_\pm (2rx + z)$. Clearly, $b_+,b_- \in (C^{0,\omega}(B_1))^{n^2}$ and satisfy \eqref{eq:a-ellip}, \eqref{eq:a-Dini} and \eqref{eq:a-k}; here the ellipticity constant, $\lambda$, and the Dini modulus of continuity, $\omega$, remain the same, and while \eqref{eq:a-k} is satisfied with $\mu:B_1 \ra [\lambda^2,\lambda^{-2}]$, defined by $\mu(x) = \kappa(2rx + z)$. 

Due to the uniform frequency condition \eqref{eq:almgren} on $u$, we have $\int_{B_1} |\nabla v|^2\,dx \leq N$ and $\int_{\partial B_1} v^2\,dx = 1$. In particular, a trace inequality (see Remark \ref{remark:meas}) yields that 
\begin{equation}\label{eq:v-L2}
\norm{v}_{L^2(B_1)} \leq c_1N,
\end{equation}
for some positive constant $c_1$ depending only on $n$. Therefore, applying Lemma \ref{lemma:compact} to $v$ for a sufficiently small number $\e$ to be determined later, we obtain a $b_+(0)$-harmonic function $g \in v^+ - \mu(0) v^- + W_0^{1,2}(B_1)$ and a small number $\bar\delta\in(0,\frac{1}{2})$ determined by $n$, $\lambda$, $\omega$ and $\e$ only, such that for any $\delta \in (0,\bar \delta]$,
\begin{equation}\label{eq:v-g-sup}
\sup_{B_{3/4}} |v^+ - \mu v^- - g| \leq C_1 \delta N, 
\end{equation} 
\begin{equation}\label{eq:q-Dg-sup}
\sup_{B_{3/4}} |Vv - \nabla g| \leq c_1 \e N, \quad\text{and}
\end{equation} 
\begin{equation}\label{eq:g-almgren}
\frac{\int_{B_1} |\nabla g|^2\,dx}{\int_{\partial B_1} g^2\,dx } \leq c_2N,
\end{equation} 
provided that $\norm{b_\pm - b_\pm(0)}_{L^\infty(B_1)} \leq \delta$; here $Vv$ is given as in Theorem \ref{theorem:int-Lip}, $C_1>0$ depends only on $n$, $\lambda$ and $\omega$, while $c_2>1$ depends only on $\lambda$. The last part yields the smallness condition on $\omega(1)$ that $\omega(1) \leq \delta$, since from the scaling relation we have $\norm{b_\pm - b_\pm(0)}_{L^\infty(B_1)} \leq \omega(2r) \leq \omega(1)$. 

Since $g$ is a $b_+(0)$-harmonic function with frequency bound \eqref{eq:g-almgren}, we can choose from Lemma 3.3 in \cite{HL94} a constant $\gamma > 0$, depending only on $n$, $\lambda$ and $N$, such that there exists a finitely collection $\{B_{r_i}(z_i)\}_{i\in I}$ of balls with $r_i \in (0,\frac{1}{2}]$ for which 
\begin{equation}\label{eq:g-cover}
B_{\frac{3}{4}} \cap \{|\nabla g| < \gamma \} \subset \bigcup_{i\in I} B_{r_i}(x_i)\quad\text{and}
\end{equation} 
\begin{equation}\label{eq:cover-rad}
\sum_{i\in I} r_i^{n-1} \leq \frac{1}{2}. 
\end{equation} 
Taking $\e$ in \eqref{eq:q-Dg-sup} small enough such that $3c_1\e N \leq \gamma$, we have 
\begin{equation}\label{eq:g-cover-re}
B_{\frac{1}{2}} \cap \{v = 0, |Vv| = 0\} \subset \bigcup_{i\in I} B_{r_i} (x_i). 
\end{equation} 

Let us select $\e$ (and $\delta$ accordingly) even smaller so that the nodal set comparison lemma (Lemma \ref{lemma:compare}) holds with $u = \frac{v}{c_3N}$, $h = \frac{g}{c_3 N}$ and $\eta = \frac{2\gamma}{3c_3N}$, where $c_3>0$ is chosen such that $\norm{v}_{L^2(B_1)} \leq c_3N$ and $\norm{g}_{C^{1,1}(B_{7/8})} \leq c_3N$; due to \eqref{eq:v-L2} and the choice that $g$ is $b_+(0)$-harmonic on $B_1$ with $g - (v^+ - b_+(0) v^-)\in W_0^{1,2}(B_1)$, $c_3 \geq c_1$ and $c_3$ must depend only on $n$ and $\lambda$. Then we have 
\begin{equation*}
\begin{split}
& H^{n-1} \left( B_{\frac{1}{2}} \cap \left\{ v = 0, |Vv| \geq \frac{2\gamma}{3} \right\} \right) \\
& \leq (1 + C_2\omega_1(\gamma))H^{n-1} \left( B_{\frac{3}{4}} \cap \left\{ g = 0, |\nabla g| \geq \frac{\gamma}{3} \right\} \right), 
\end{split} 
\end{equation*} 
where $C_2>0$ depends only on $n$, $\lambda$ and $\omega$. Taking $\e$ so as to satisfy $3c_3\e N \leq \gamma$, we can deduce from \eqref{eq:q-Dg-sup} and\eqref{eq:g-cover} yields that 
\begin{equation}\label{eq:v-meas}
H^{n-1} \left( B_{\frac{1}{2}} \cap \{ v = 0 \} \setminus \bigcup_{i \in I} B_{r_i}(x_i) \right) \leq C_3 N, 
\end{equation} 
with $C_3$ depending only on $n$, $\lambda$ and $\omega$, where in the last inequality we used Theorem 3.1 in \cite{Lin91}, which yields $H^{n-1}(B_{\frac{3}{4}}\cap \{g=0\}) \leq c_4N$ for some positive constant $c_4$ depending only on $n$ and $\lambda$. 

Tracking down the dependence on $\e$, we see that it depends only on $n$, $\lambda$, $\omega$ and $N$. Now we choose $\delta$, involved in the smallness condition of $\omega(1)$ above so that \eqref{eq:v-g-sup}, \eqref{eq:q-Dg-sup} and \eqref{eq:g-almgren} hold, so as to satisfy $C_1\delta \leq c_3\e$. Clearly, $\delta$ is determined by the same set of parameters as $\e$. In particular, these constants do not vary upon $l$, the index for the iteration step.  

Now let us return to $B = B_r(z)\in \cB_{l-1}$, define $\cB_l^B$ by the finite collection $\{B_{r_i r}(2rx_i + z)\}_{i\in I}$ and then $\cB_l$ by $\bigcup_{B \in \cB_{l-1}} \cB_l^B$. To this end, one can easily check from \eqref{eq:g-cover-re}, \eqref{eq:cover-rad} and \eqref{eq:v-meas} that initial hypotheses \eqref{eq:Bk-1}, \eqref{eq:Bk-2} and respectively \eqref{eq:Bk-3} are satisfied with $k = l$. Therefore, one can inductively construct the finite covering of $\{u=0\}\cap B_{1/2}$, and finally proving the desired $H^{n-1}$-measure estimate. Let us omit the rest of the argument, which is identical with that of Theorem 3.1 in \cite{HL94}, and finish the proof. 
\end{proof}

\noindent
{\bf Acknowledgement.} 
This project was supported by Knut and Alice Wallenberg Foundation. The author would also like to thank H. Shahgholian and J. Andersson for valuable discussions on this topic.


\end{document}